\documentclass[11pt]{amsart}
\usepackage{amsmath}
\usepackage{amsthm}
\usepackage{amsfonts}
\usepackage{latexsym}
\usepackage{graphics}
\usepackage{epsfig}
\usepackage{graphicx}
\usepackage[colorlinks=true]{hyperref}

\newtheorem{theorem}{Theorem}[section]

\newtheorem{lemma}[theorem]{Lemma}
\newtheorem{proposition}{Proposition}

\renewcommand{\epsilon}{\varepsilon}
\renewcommand{\rho}{\varrho}
\renewcommand{\phi}{\varphi}
\newcommand{\DS}{\displaystyle}
\newcommand{\N}{{\mathbb N}}
\newcommand{\Z}{{\mathbb Z}}
\newcommand{\R}{{\mathbb R}}
\newcommand{\cA}{{\mathcal A}}
\newcommand{\cB}{{\mathcal B}}
\newcommand{\cG}{{\mathcal G}}
\newcommand{\cH}{{\mathcal H}}
\newcommand{\cL}{{\mathcal L}}
\newcommand{\cP}{{\mathcal P}}
\newcommand{\cX}{{\mathcal X}}
\newcommand{\cY}{{\mathcal Y}}
\begin{document}

\title[Equilibrium Validation Based on Sobolev Embeddings]
      {Equilibrium Validation in Models for Pattern Formation
       Based on Sobolev Embeddings}

\author{Evelyn Sander}
\address{Department of Mathematical Sciences\\
        George Mason University\\
        Fairfax, VA 22030, USA}
\email{esander@gmu.edu}

\author{Thomas Wanner}
\address{Department of Mathematical Sciences\\
        George Mason University\\
        Fairfax, VA 22030, USA}
\email{twanner@gmu.edu}


\maketitle

\begin{abstract}
In the study of equilibrium solutions for partial differential equations there
are so many equilibria that one cannot hope to find them all. Therefore one
usually concentrates on finding individual branches of equilibrium solutions. On
the one hand, a rigorous theoretical understanding of these branches is ideal
but not generally tractable. On the other hand, numerical bifurcation searches are
useful but not guaranteed to give an accurate structure,  in that they could miss
a portion of a branch or find a spurious branch where none exists. In a series
of recent papers, we have aimed for a third option. Namely, we have developed a
method of computer-assisted proofs to prove both existence and isolation of
branches of equilibrium solutions. In the current paper, we extend these techniques
to the Ohta-Kawasaki model for the dynamics of diblock copolymers in dimensions
one, two, and  three, by giving a detailed description of the analytical
underpinnings of the method. Although the paper concentrates on applying
the method to the Ohta-Kawasaki model, the functional analytic approach
and techniques can be generalized to other parabolic partial differential
equations.
\end{abstract}

\section{Introduction}
\label{secintro}
The goal of this paper is to present the theoretical underpinnings for
computer-assisted branch validation using functional analytic techniques 
including the constructive implicit function theorem and Neumann series methods, 
such that pointwise estimates result in solution branch validation. 
While the individual proof techniques presented here are not novel,
we present this approach in a modular way such that it is flexible, adaptable, 
and as computationally feasible as possible in more than one space dimension. 
In particular, we apply this methodology in the case of the Ohta--Kawasaki
model for diblock copolymers~\cite{ohta:kawasaki:86a}. Diblock copolymers 
are formed by the chemical reaction of two linear polymers (known as
blocks) which contain different monomers. Whenever the blocks are
thermodynamically incompatible, the blocks are forced to separate
after the reaction, but since the blocks are covalently bonded they
cannot separate on a macroscopic scale. The competition between these
long-range and short-range forces causes {\em microphase separation},
resulting in pattern formation on a mesoscopic scale.

We study the Ohta-Kawasaki equation in the case of homogeneous Neumann
boundary conditions on rectilinear domains $\Omega$ in dimensions one,
two, and three, which is given by 
\begin{eqnarray*}
  w_t &=& -\Delta ( \Delta w + \lambda f(w) ) - \lambda \sigma (w - \mu)
  \quad\mbox{ in } \Omega \;, \\[1.5ex]
  & & \frac{ \partial w }{\partial \nu } =
    \frac{\partial (\Delta w)}{\partial \nu }= 0
    \quad\mbox{ on } \partial \Omega \; . 
\end{eqnarray*}
The notation~$\nu$ denotes the unit outward normal on the boundary
of~$\Omega$ --- corresponding to homogeneous Neumann boundary conditions.
The quantity $w(t,x)$ is the local average density of the two blocks.
The parameter $\mu$ is the space average of~$w$, meaning it is a measure
of the relative total proportion of the two polymers, which we tersely
refer to as the {\em mass} of the system.  The equation obeys a mass
conservation,  implying that~$\mu$ is time-invariant. A large value of
parameter~$\lambda$  corresponds to a large short-range repulsion,
while a large value of the parameter~$\sigma$ corresponds to large
long-range elasticity forces. We refer the reader
to~\cite{johnson:etal:13a} for a detailed description of
how~$\lambda$ and~$\sigma$ are defined. 
The nonlinear function $f: \R \to \R$ is often assumed to be $f(u) = u-u^3$, but 
the results in this paper still apply as long as $f$ is a  $C^2$-function. 
Finally, note that the second
boundary condition is necessary since this is a fourth order equation. 
In this paper, we focus on equilibrium solutions $w = w(x)$. 

For notational convenience, we reformulate our equation slightly. 
For a solution~$w$ of the diblock copolymer equation, we define
$u = w - \mu$. Since the space average of~$w$ is~$\mu$, the average
of the shifted function~$u$ is zero. Therefore the equilibrium equation
becomes
\begin{eqnarray} \label{eqn:dbcp}
-\Delta ( \Delta u + \lambda f(u+\mu) ) - \lambda \sigma u  &=& 0 \quad\mbox{ in } \Omega \;, \nonumber\\[1.5ex]
\frac{ \partial u }{\partial \nu }=  \frac{ \partial (\Delta u) }{\partial \nu }&=& 0 \quad\mbox{ on } \partial \Omega \;,  \\[1.5ex]
\int_\Omega u \;dx &=& 0 \nonumber \;. 
\end{eqnarray}
We will use this version of the equation  for the rest of the paper. We focus
on solutions to this equation as we vary any of the three parameters: the
degree of short-range repulsion $\lambda$, the mass $\mu$, and the degree of 
long-range elasticity  $\sigma$. Our main goal is to establish  bounds that
make it possible to use a functional analytic approach to rigorous validation
using the point of view of the constructive implicit function theorem which we
have already developed in previous work~\cite{sander:wanner:16a, wanner:16a,
wanner:17a, wanner:18a}. Our bounds are developed mostly using theoretical
techniques, but in the case of Sobolev embeddings, the bounds themselves are
developed using computer-assisted means. This method is designed for validated
continuation of branches of solutions which depend on a parameter, in the
spirit of the numerical method of pseudo-arclength continuation, such as
seen in the software packages AUTO~\cite{doedel:81a} and Matcont~\cite{matcont}.
Successive application of this theorem allows us to validate branches of
equilibrium solutions by giving precise bounds on both the branch
approximation error and isolation. This is much more powerful than
only validating individual solutions along a branch, since it allows
us to guarantee that a set of solutions lie along the same connected
branch component. 

In order to establish what is new in this paper, we give a brief discussion
of previous results. A number of papers have previously considered numerical
computation of bifurcation diagrams for the Ohta-Kawasaki and Cahn-Hilliard 
equations, such as for example~\cite{choksi:etal:11a, choksi:peletier:09a,
choksi:ren:03a, choksi:ren:05a, desi:etal:11a, johnson:etal:13a, maier:etal:08a,
maier:etal:07a}. There are also several decades of results on computer validation
for dynamical systems and differential equations solutions which combine fixed
point arguments and interval arithmetic; see for example~\cite{arioli:koch:10a,
day:etal:07b, gameiro:etal:08a, plum:95a, plum:96a, plum:09a, rump:10a, wanner:16a,
wanner:17a}. A constructive implicit function theorem was formulated in the work
of Chierchia~\cite{chierchia:03a}. Our approach follows most closely the work of
Plum~\cite{nakao:etal:19a, plum:95a, plum:96a,plum:09a}, in which functional
analytic approaches are given for establishing needed apriori bounds. Such
methods have also been applied by Yamamoto~\cite{yamamoto:98a, yamamoto:etal:11a}.
In our previous work on the constructive implicit function theorem, our goal has
been to give a systematic procedure for adapting these works to the context of
parameter continuation. There are several papers that have already considered
rigorous validation of parameter-dependent solutions for the Ohta-Kawasaki 
model~\cite{cai:watanabe:19a, cyranka:wanner:18a, lessard:sander:wanner:17a,
sander:wanner:16a, vandenberg:williams:17a, vandenberg:williams:p19b,
vandenberg:williams:19a, wanner:16a, wanner:17a, wanner:18a}. Many of these
papers also include methods of bounding the terms in a generalized Fourier
series, and the estimates on the tail. However, it was necessary to make
quite substantial ad hoc calculations in order to establish needed bounds
before it is possible to proceed with numerical validation.

Our goal in the current paper is to establish a set of flexible
bounds on the size of the inverse of the derivative, the required
truncation dimension, Lipschitz bounds on the equations with respect to all
parameters, as well as constructive Sobolev embedding constant bounds for
comparison to the $L^\infty$-norm, meaning that equilibrium verifications
along branch segments can be done without having to resort to ad hoc
calculations which crucially depend on the specific nonlinearity.
More precisely, we obtain the following:
\begin{itemize}
\item The approach of this paper derives general estimates that work in one,
two, and three space dimensions, and under the natural homogeneous Neumann
boundary conditions. This is in contrast to~\cite{cai:watanabe:19a}
and~\cite{wanner:17a}, which only considered the case of one-dimensional
domains, or to~\cite{vandenberg:williams:p19b, vandenberg:williams:19a}, which
considered the three-dimensional case only under periodic boundary conditions
and symmetry constraints.
\item Our approach uses the natural functional analytic setting for the
diblock copolymer evolution equation, which is based on the Sobolev space
of twice weakly differentiable functions. This is in contrast
to~\cite{vandenberg:williams:p19b, vandenberg:williams:19a}, which seek 
the equilibria in spaces of analytic functions.
\item As part of our approach, we obtain accurate upper bounds for the
operator norm of the inverse of the diblock copolymer Fr\'echet derivative.
For this estimate, we use the natural Sobolev norms of the underlying
problem. In contrast to~\cite{kinoshita:etal:19a,
watanabe:etal:19a, watanabe:etal:16a} our method is based on Neumann
series.
\end{itemize}
Throughout this paper, we focus on the theoretical underpinnings
which allow one to apply the constructive implicit function
theorem~\cite{sander:wanner:16a}. Due to space constraints, we leave
the practical application of these results to path-following with
slanted boxes as in~\cite{sander:wanner:16a}, as well as extensions
to pseudo-arclength continuation, for future work. Nevertheless, while
this paper is focussed only on the Ohta-Kawasaki model, the general
approach can be used for other parabolic partial differential
equations as well. 

The remainder of this paper is organized as follows. In
Section~\ref{sec:tbd}, we introduce the necessary functional analytic
framework, while Section~\ref{subsec:die} is devoted to finding bounds
on the operator norm of the inverse of the linearized operator. After
that, Section~\ref{sec:cift} establishes Lipschitz bounds on the diblock
copolymer operator for continuation with respect to any of the three
parameters~$\lambda$, $\sigma$, and~$\mu$, before in Section~\ref{sec:eg}
we give a brief numerical illustration of how this method rigorously
establishes a variety of equilibrium branch pieces for the Ohta-Kawasaki
model in multiple dimensions. Finally, in Section~\ref{sec:theend} we
wrap up with conclusions and future plans.
\section{Basic definitions and setup}
\label{sec:tbd}
In this section, we establish notation and crucial auxiliary bounds. In
Section~\ref{sec:ciftstatement} we recall the constructive implicit function
theorem, before in Section~\ref{sec:fun} we define the function spaces that will
be used in our computer-assisted proofs. These spaces are particularly adapted for
the use with Fourier series expansions to represent functions with Neumann
boundary conditions and zero average. In Section~\ref{sec:sob}, we
collect a set of Sobolev embedding results giving precise rigorous
bounds on the similarity constants for passing between equivalent norms
on these function spaces. Finally, in Section~\ref{sec:proj} we 
introduce the necessary finite-dimensional spaces and associated
projection operators that are used in our computer-assisted proofs.
\subsection{The constructive implicit function theorem}
\label{sec:ciftstatement}
In this section we state a constructive implicit function theorem that makes
it possible to validate a branch of solutions changing with respect to a
parameter. This theorem appears in~\cite{sander:wanner:16a}, where we
demonstrated the validation of solutions for the lattice Allen-Cahn equation. 
The theorem is based on previous work of Plum~\cite{plum:09a} and
Wanner~\cite{wanner:17a}. To put this in context, our overarching goal
is to find a connected curve of values~$(\alpha,x)$ in the zero set for a 
specific nonlinear operator~$\cG(\alpha,x)$. In this paper, the zero set
consists of the equilibria of the Ohta-Kawasaki equation. Starting at a point
for which the operator~$\cG$ is close to zero, we use the theorem as the
iterative step in a validated continuation. That is, we iteratively validate
small portions along the solution curve, each time using the constructive
implicit function theorem which is stated below. We also validate that these
portions combine to create a piece of a single connected solution curve, 
and show that it is isolated from any other branch of the solution curve. 
Rather than getting bogged down in the details of the iterative process, we 
first concentrate on the single iterative step and the estimates needed in
order to perform it. Specifically, we consider solutions to the equation 
\begin{equation} \label{nift:nleqn}
  \cG(\alpha,x) = 0 \; ,
\end{equation}
where $\cG : \cP \times \cX \to \cY$ is a Fr\'echet differentiable
nonlinear operator between two Banach spaces~$\cX$ and~$\cY$, and
the parameter~$\alpha$ is taken from a Banach space~$\cP$. The
norms on these Banach spaces are denoted by~$\|\cdot\|_\cP$,
$\|\cdot\|_\cX$, and~$\|\cdot\|_\cY$, respectively. One possible
choice of~$\cG$ would be to directly use the nonlinear operator
associated with~(\ref{eqn:dbcp}), but this is not a numerically
viable option for validation of a branch of solutions. Instead we
will introduce an extended system which gives a validated version of
pseudo-arclength continuation. The system contains not only the
Ohta-Kawasaki model equilibrium equation, but is in a way designed
to optimize the needed number of validation steps. 

In order to present the constructive implicit function theorem in
detail, we begin by making the following hypotheses. 
For the classical implicit function theorem, the existence of constants 
satisfying the hypotheses given below is sufficient. In contrast, since  we wish to 
use a computer assisted proof to validate existence of equilibria with
specified error bounds, we require
explicit values for each of the constants in (H1)--(H4). 

\begin{itemize}
\item[(H1)] Unlike the traditional implicit function theorem, we assume 
only an approximate solution to the equation. That is, assume that we
are given a pair~$(\alpha^*,x^*) \in \cP \times \cX$ which is an
approximate solution of the nonlinear problem~(\ref{nift:nleqn}). 
More precisely, the residual of the nonlinear operator~$\cG$ at the
pair~$(\alpha^*,x^*)$ is small, i.e., there exists a constant
$\rho > 0$ such that
\begin{displaymath}
  \left\| \cG(\alpha^*,x^*) \right\|_\cY \le \rho \; .
\end{displaymath}
\item[(H2)] Assume that the operator~$D_x\cG(\alpha^*,x^*)$
is invertible and not very close to being singular. That is, 
the Fr\'echet derivative~$D_x\cG(\alpha^*,x^*) \in \cL(\cX,\cY)$,
where~$\cL(\cX,\cY)$ denotes the Banach space of all bounded
linear operators from~$\cX$ into~$\cY$, is one-to-one and
onto, and its inverse~$D_x\cG(\alpha^*,x^*)^{-1} : \cY \to \cX$
is bounded and satisfies
\begin{displaymath}
  \left\| D_x\cG(\alpha^*,x^*)^{-1} \right\|_{\cL(\cY,\cX)} \le K
  \; ,
\end{displaymath}
where~$\| \cdot \|_{\cL(\cY,\cX)}$ denotes the operator norm
in~$\cL(\cY,\cX)$. 
\item[(H3)] For~$(\alpha,x)$ close to~$(\alpha^*,x^*)$,
the Fr\'echet derivative~$D_x\cG(\alpha,x)$ is locally
Lipschitz continuous in the following sense. There exist
positive real constants~$L_1$, $L_2$, $\ell_x$,
and~$\ell_\alpha \ge 0$ such that for all pairs $(\alpha,x)
\in \cP \times \cX$ with $\| x - x^* \|_\cX \le \ell_x$ and
$\|\alpha - \alpha^*\|_\cP \le \ell_\alpha$ we have
\begin{displaymath}
  \left\| D_x\cG(\alpha,x) -
    D_x\cG(\alpha^*,x^*) \right\|_{\cL(\cX,\cY)} \le
    L_1 \left\| x - x^* \right\|_\cX +
    L_2 \left\|\alpha - \alpha^* \right\|_\cP \; .
\end{displaymath}
To verify this condition, as well as the next one, we will 
give specific Lipschitz bounds on the Ohta-Kawasaki operator. 
We will then show the precise way to combine these bounds in 
order to get the constants $L_k$.
\item[(H4)] For~$\alpha$ close to~$\alpha^*$, the
Fr\'echet derivative~$D_\alpha \cG(\alpha,x^*)$ satisfies
a Lipschitz-type bound. More precisely, there exist positive
real constants~$L_3$ and~$L_4$, such that for all $\alpha \in \cP$
with $\|\alpha - \alpha^*\|_\cP \le \ell_\alpha$ one has
\begin{displaymath}
  \left\| D_\alpha \cG(\alpha,x^*) \right\|_{\cL(\cP,\cY)} \le
    L_3 + L_4 \left\| \alpha - \alpha^* \right\|_\cP \; ,
\end{displaymath}
where~$\ell_\alpha$ is the constant that was chosen in~(H3).
\end{itemize}
Keeping these hypotheses in mind, the constructive implicit
function theorem can then be stated as follows.
\begin{theorem}[Constructive Implicit Function Theorem]
\label{nift:thm}
Let~$\cP$, $\cX$, and~$\cY$ be Banach spaces, suppose that the
nonlinear operator $\cG : \cP \times \cX \to \cY$
is Fr\'echet differentiable, and assume that the
pair~$(\alpha^*,x^*) \in \cP \times \cX$ satisfies
hypotheses~(H1), (H2), (H3), and~(H4).
Finally, suppose that
\begin{equation} \label{nift:thm1}
  4 K^2 \rho L_1 < 1
  \qquad\mbox{ and }\qquad
  2 K \rho < \ell_x \; .
\end{equation}
Then there exist pairs of constants~$(\delta_\alpha,\delta_x)$ with
$0 \le \delta_\alpha \le \ell_\alpha$ and $0 < \delta_x \le \ell_x$,
as well as
\begin{equation} \label{nift:thm2}
  2 K L_1 \delta_x + 2 K L_2 \delta_\alpha \le 1
  \qquad\mbox{ and }\qquad
  2 K \rho + 2 K L_3 \delta_\alpha + 2 K L_4 \delta_\alpha^2
    \le \delta_x  \; ,
\end{equation}
and for each such pair the following holds. For every~$\alpha \in \cP$
with $\|\alpha - \alpha^*\|_\cP \le \delta_\alpha$ there exists a uniquely
determined element~$x(\alpha) \in \cX$ with $\| x(\alpha) - x^* \|_\cX
\le \delta_x$ such that $\cG(\alpha, x(\alpha)) = 0$.
In other words, if we define
\begin{displaymath}
  \cB_\delta^\cX = \left\{ \xi \in \cX \; : \;
    \left\| \xi - x^* \right\|_\cX \le \delta \right\}
  \quad\mbox{ and }\quad
  \cB_\delta^\cP = \left\{ p \in \cP \; : \;
    \left\| p - \alpha^* \right\|_\cP \le \delta \right\}
  \; ,
\end{displaymath}
then all solutions of the nonlinear problem $\cG(\alpha,x)=0$ in the
set $\cB_{\delta_\alpha}^\cP \times \cB_{\delta_x}^\cX$ lie on the graph
of the function $\alpha \mapsto x(\alpha)$.  In addition, the following
two statements are satisfied.
\begin{itemize}
\item For all pairs $(\alpha,x) \in \cB_{\delta_\alpha}^\cP
\times \cB_{\delta_x}^\cX$ the Fr\'echet derivative~$D_x\cG(\alpha,x)
\in \cL(\cX,\cY)$ is a bounded invertible linear operator, whose
inverse is in~$\cL(\cY,\cX)$.
\item If the mapping~$\cG : \cP \times \cX \to \cY$ is $k$-times
continuously Fr\'echet differentiable, then so is the solution
function $\alpha \mapsto x(\alpha)$.
\end{itemize}
\end{theorem}
Throughout the remainder of this paper, we concentrate on finding
computationally accessible versions of hypotheses~(H2), (H3), and~(H4)
for the Ohta-Kawasaki model.  
\subsection{Function spaces}
\label{sec:fun}
Throughout this paper, we let $\Omega = (0,1)^d$ denote the unit cube
in dimension $d = 1,2,3$, and define the constants
\begin{displaymath}
  c_0 = 1
  \quad\mbox{ and }\quad
  c_\ell = \sqrt{2}
  \quad\mbox{ for }\quad
  \ell \in \N \; .
\end{displaymath}
If $k \in \N_0^d$ denotes an arbitrary multi-index of the form
$k = (k_1,\dots,k_d)$, then let
\begin{displaymath}
  c_k = c_{k_1} \cdot \ldots \cdot c_{k_d} \; .
\end{displaymath}
If we then define
\begin{equation}
\label{eqn:phik}
  \phi_k(x) = c_k \prod_{i=1}^d \cos ( k_i \pi x_i )
  \quad\mbox{ for all }\quad
  x = (x_1,\ldots,x_d) \in \Omega \; , 
\end{equation}
then the function collection $\{ \phi_k \}_{k \in \N_0^d}$ forms a complete orthonormal
basis for the space~$L^2(\Omega)$. Any measurable and square-integrable function
$u : \Omega \to \R$ can be written in terms of its Fourier cosine series
\begin{equation} \label{eqn:ucossum}
  u(x) = \sum_{k \in \N_0^d} \alpha_k \phi_k(x) \; ,
\end{equation}
where $\alpha_k \in \R$ are the Fourier coefficients of~$u$. Finally, we
define
\begin{displaymath}
  |k| = (k_1^2 + \dots + k_d^2)^{1/2}
  \quad\mbox{ and }\quad
  |k|_\infty = \max( k_1,  \dots , k_d ) \; .
\end{displaymath}
Each function $\phi_k(x)$ is an eigenfunction of the negative Laplacian. The corresponding eigenvalue 
is given by $\kappa_k$, defined via the equation 
\begin{displaymath}
  -\Delta \phi_k(x) = \kappa_k \phi_k(x)
  \qquad\mbox{ with }\qquad
  \kappa_k = \pi^2 \left(k_1^2 + k_2^2 + \dots + k_d^2\right) =
    \pi^2 |k|^2 \; .
\end{displaymath}
A straightforward direct computation shows that each~$\phi_k(x)$ satisfies
the homogeneous Neumann boundary condition $\partial \phi_k /\partial \nu = 0$. 
In addition, as a result of being an eigenfunction of~$-\Delta$, each function $\phi_k(x)$ 
also satisfies the second boundary condition in~(\ref{eqn:dbcp}), since the
identity $\partial (\Delta \phi_k) /\partial \nu = -\kappa_k  \partial
\phi_k /\partial \nu = 0$ holds. Therefore any finite Fourier series as
above automatically satisfies both boundary conditions of the diblock copolymer
equation.

Based on our construction, the family $\{ \phi_k \}_{k \in \N_0^d}$ is a
complete orthonormal basis for the space~$L^2(\Omega)$. Thus, if~$u$ is
given as in~(\ref{eqn:ucossum}) one can easily see that
\begin{displaymath}
  \| u \|_{L^2} = \left( \sum_{k \in \N_0^d} \alpha_k^2 \right)^{1/2} \; .
\end{displaymath}
For our application to the diblock copolymer model, we need to work
with suitable subspaces of the Sobolev spaces~$H^k(\Omega) = W^{k,2}(\Omega)$,
see for example~\cite{adams:fournier:03a}. These subspaces have to
reflect the required homogeneous Neumann boundary conditions and they can be
introduced as follows.
For $\ell \in \N$ consider the space 
\begin{displaymath}
  \cH^\ell =
    \left\{ u = \sum_{k \in \N_0^d} \alpha_k \phi_k : \| u \|_{\cH^\ell}
    < \infty \right\} \;,
\end{displaymath}
where
\begin{displaymath}
  \|u\|_{\cH^\ell} = \left( \sum_{k \in \N_0^d} \left( 1 +
    \kappa_k^\ell \right) \alpha_k^2 \right)^{1/2} \; . 
\end{displaymath}
One can easily verify that this is equivalent to the definition
\begin{displaymath}
  \|u\|^2_{\cH^\ell} =
  \|u \|_{L^2}^2 + \left\| (-\Delta)^{\ell/2} u \right\| ^2_{L^2} \;,
\end{displaymath}
where~$\| \cdot \|_{L^2}$ denotes the standard $L^2(\Omega)$-norm
on the domain~$\Omega$ as mentioned above, and the fractional Laplacian 
for odd~$\ell$ is defined using the spectral definition. 
We note that we have incorporated the boundary conditions 
of~(\ref{eqn:dbcp}) into our definition of the spaces
$\cH^\ell$. For example, 
\begin{eqnarray*}
\cH^1 &=& H^1(\Omega), \\
\cH^2 &=& \left\{ u \in H^2(\Omega): \frac{\partial u}{\partial \nu} =0  \right\} \, \quad\mbox{, and } \\ 
\cH^4 &=& \left\{ u \in H^4(\Omega): \frac{\partial u}{\partial \nu} = \frac{\partial \Delta u}{\partial \nu} =0  \right\} \, ,
\end{eqnarray*}
where the boundary conditions in the second and third equations are considered in the sense of the trace operator.
The first identity follows as a special case from the results
in~\cite{huseynov:shykhammedov:12a, muradov:salmanov:14a}, the second 
identity has been established in~\cite[Lemma~3.2]{maier:wanner:00a}, and also the third
identity can be verified as in~\cite[Lemma~3.2]{maier:wanner:00a}.
For the sake of simplicity
we further define $\cH^0 = L^2(\Omega)$.

While the spaces~$\cH^\ell$ incorporate the boundary conditions
of~(\ref{eqn:dbcp}), recall that we have reformulated the diblock
copolymer equation in such a way that solutions satisfy the integral
constraint $\int_\Omega u \; dx = 0$, since the case of nonzero average
has been absorbed into the placement of the parameter~$\mu$. In order
to treat this additional constraint, we therefore need to restrict the
spaces~$\cH^\ell$ further. Consider now an {\em arbitrary integer\/}~$\ell
\in \Z$ and define the space
\begin{equation} \label{defhlbar}
  \overline{\cH}^\ell =
  \left\{ u = \sum_{k \in \N_0^d, \; |k|>0} \alpha_k \phi_k
    \; : \; \| u \|_{\overline{\cH}^\ell} < \infty \right\} \;, 
\end{equation}
where we use the modified norm
\begin{equation} \label{defhlbarnorm}
  \| u \|_{\overline{\cH}^\ell} =
  \left( \sum_{k \in \N_0^d, \; |k|>0} \kappa_k^\ell \alpha_k^2
    \right)^{1/2} \;. 
\end{equation}
Notice that for $\ell = 0$ this definition reduces to the subspace
of~$L^2(\Omega)$ of all functions with average zero equipped with its
standard norm, since we removed the constant basis function from the
Fourier series. For $\ell > 0$ one can easily see that
$\overline{\cH}^\ell \subset \cH^\ell$, and  that the new norm
is equivalent to our norm on~$\cH^\ell$. We still need to shed
some light on the new definition~(\ref{defhlbar}) for
negative integers $\ell < 0$. In this case, the series
in~(\ref{eqn:ucossum}) is interpreted formally, i.e., the element
$u \in \overline{\cH}^\ell$ for $\ell < 0$ is identified with the
sequence of its Fourier coefficients. Moreover, one can easily see
that in this case~$u$ acts as a bounded linear functional
on~$\overline{\cH}^{-\ell}$. In fact, for all~$\ell < 0$ the 
space~$\overline{\cH}^\ell$ can be considered as a subspace of
the negative exponent Sobolev space~$H^\ell(\Omega) =
W^{\ell,2}(\Omega)$, see again~\cite{adams:fournier:03a}.
Finally, for every $\ell \in \Z$ the space~$\overline{\cH}^\ell$ is
a Hilbert space with inner product
\begin{displaymath}
  (u,v)_{\overline{\cH}^\ell} =
  \sum_{k \in \N_0^d, \; |k|>0} \kappa_k^\ell \alpha_k \beta_k \;, 
\end{displaymath}
where
\begin{displaymath}
  u = \sum_{k \in \N_0^d, \; |k|>0} \alpha_k \phi_k
    \in \overline{\cH}^\ell
  \qquad\mbox{ and }\qquad
  v = \sum_{k \in \N_0^d, \; |k|>0} \beta_k \phi_k
    \in \overline{\cH}^\ell \;. 
\end{displaymath}
The above spaces form the functional analytic backbone of this
paper, and they allow us to reformulate the equilibrium problem 
for~(\ref{eqn:dbcp}) as a zero finding problem. Note first, however,
that the functions~$\phi_k$ can also be used to obtain an orthonormal
basis in~$\overline{\cH}^\ell$. In fact, we only have to drop the 
constant function~$\phi_0$ and apply the following rescaling.
\begin{lemma} \label{lemma:orthoghell}
The set $\left\{ \kappa_k^{-\ell/2} \phi_k(x) \right\}_{k \in \N_0^d,
\; |k|>0}$
forms a complete orthonormal set for the Hilbert space~$\overline{\cH}^\ell$.
\end{lemma}
We close this section by briefly showing how the diblock copolymer
equilibrium problem can be stated as a zero set problem in our functional
analytic setting. For this, consider the operator
\begin{displaymath}
  F : \R^3 \times X \to Y \; ,
  \qquad\mbox{ with }\qquad
  X = \overline{\cH}^2
  \quad\mbox{ and }\quad
  Y = \overline{\cH}^{-2} \; , 
\end{displaymath}
which is defined as
\begin{equation} \label{eqn:deffoperator}
  F(\lambda,\sigma,\mu, u) =
  -\Delta \left( \Delta u + \lambda f(u + \mu) \right)
    - \lambda \sigma u \; . 
\end{equation}
The problem is now formulated weakly, and in particular, the second boundary 
condition $\partial (\Delta u) / \partial \nu = 0$ is no longer explicitly
stated in this weak formulation. Note, however, that the first boundary 
condition $ \partial u / \partial \nu = 0$ has been incorporated into the
space~$X = \overline{\cH}^2$.
The fact that $f$ is $C^2$ is sufficient to guarantee that the function~$F$
maps $X$ to $Y$, since we only consider domains up to dimension three. 
Then for fixed parameters, an equilibrium solution~$u$ to the diblock
copolymer equation~(\ref{eqn:dbcp}) is a function which satisfies the
identity $F(\lambda,\sigma,\mu, u) = 0$. Moreover, the Fr\'echet derivative
of the operator~$F$ with respect to~$u$ at this equilibrium is given by
\begin{equation} \label{eqn:deffrechetderivative}
  D_u F(\lambda, \sigma, \mu, u) [v] =
  - \Delta \left( \Delta v + \lambda f'(u + \mu)v \right)
    - \lambda \sigma v \; . 
\end{equation}
In our formulation, the boundary and integral conditions which are
part of~(\ref{eqn:dbcp}) have been incorporated into the choice of
the domain $X = \overline{\cH}^2$ of the nonlinear operator~$F$. 
\subsection{Constructive Sobolev embedding and Banach algebra constants}
\label{sec:sob}
For classical Sobolev embedding theorems, it is sufficient to write
statements such as ``the Sobolev space~$\cH^2$ can be continuously embedded
into~$L^\infty(\Omega)$,'' without worrying about the specific constants needed
to do so. However, for the purpose of computer-assisted proofs, such
statements are insufficient. Instead we need  specific numerical bounds
to compare the norms of a function or product of functions when
considered in different spaces. Parallel to  the  name constructive
implicit function theorem, we refer to the bounds on the constants as
{\em constructive Sobolev embedding constants\/}. In addition, we will
need a {\em constructive Banach algebra estimate\/} on the relationship between
$\| u v \|_{\cH^2}$ and the product $\|u\|_{\cH^2} \|v\|_{\cH^2}$. In
particular, we require the exact values of $C_m$, $\overline{C}_m$, and
$C_b$ in one, two, and three dimensions given in the following
equations:
\begin{eqnarray}
  \|u \|_\infty &\le& C_m \; \| u \|_{\cH^2} \;,
    \qquad\qquad\mbox{ for all } u \in \cH^2 \;,
    \nonumber \\[1ex]
  \| u \|_\infty &\le& \overline{C}_m \; \| u \|_{\overline{\cH}^2} \;,
    \qquad\qquad\mbox{ for all } u \in \overline{\cH}^2
    \label{eqn:cmcb} \;,\\[1.5ex]
  \| u v \|_{\cH^2} &\le& C_b \;  \| u \|_{\cH^2} \| v \|_{\cH^2}
    \;, \quad\;\;\;\mbox{ for all } u,v \in \cH^2 \nonumber \;. 
\end{eqnarray}
The values of~$C_m$ and~$C_b$ in dimensions~$1$, $2$, and~$3$ were
established in~\cite{wanner:18b} using rigorous computational techniques.
The values of~$\overline{C}_m$ can be obtained by adapting the approach
in this paper, as outlined in the next lemma. Table~\ref{table1} 
summarizes the values of all necessary constants.
\begin{table}
\begin{center}
    \begin{tabular}{|c||c|c|c|}
    \hline
    Dimension $d$ & $1$ & $2$ & $3$ \\ \hline\hline
    Sobolev Embedding Constant $C_m$ & $1.010947$ &
      $1.030255$ & $1.081202$  \\ \hline
    Sobolev Embedding Constant $\overline{C}_m$ &
      $0.149072$ & $0.248740$ & $0.411972$ \\ \hline
    Banach Algebra Constant $C_b$ &
      $1.471443$ & $1.488231$ & $1.554916$ \\ \hline  
    \end{tabular}
    \vspace*{0.3cm}
    \caption{\label{table1} 
    These values are rigorous upper bounds for the embedding
             constants in~(\ref{eqn:cmcb}).}
\end{center}
\end{table}
\begin{lemma}[Sobolev embedding for the zero mass case]
\label{lemma:sobzeromass}
For all functions $u \in \overline{\cH}^2$ we have the estimate
\begin{equation} \label{lemma:sobzeromass1}
  \| u \|_\infty \quad\le\quad
  \| u \|_{\overline{\cH}^2} \cdot \left( \sum_{k \in \N_0^d, \; |k|>0}
    c_k^2 \kappa_k^{-2} \right)^{1/2} \quad\le\quad
  \overline{C}_m \| u \|_{\overline{\cH}^2} \; ,
\end{equation}
where the value of the constant~$\overline{C}_m$ is given
in Table~\ref{table1}.
\end{lemma}
\begin{proof}
Suppose that $u \in \overline{\cH}^2$ is given by
$u = \sum_{k \in \N_0^d, \; |k|>0} \alpha_k \phi_k$.
According to the definition of the functions~$\phi_k$
we have $\| \phi_k \|_\infty = c_k$, which immediately implies
for all $x \in \Omega$ the estimate
\begin{eqnarray*}
  |u(x)| & \le & \sum_{k \in \N_0^d, \; |k|>0} \left| \alpha_k \right|
    \, \left| \phi_k(x) \right|
    \; \le \;
    \sum_{k \in \N_0^d, \; |k|>0} \left| \alpha_k \right| c_k
    \; = \; 
    \sum_{k \in \N_0^d, \; |k|>0} \left| \alpha_k \right| 
    \kappa_k \cdot \frac{c_k}{\kappa_k} \\[2ex]
  & \le & \left( \sum_{k \in \N_0^d, \; |k|>0} \alpha_k^2 
    \kappa_k^2 \right)^{1/2} \cdot
    \left( \sum_{k \in \N_0^d, \; |k|>0} c_k^2 \kappa_k^{-2} \right)^{1/2} \; ,
\end{eqnarray*}
and together with~(\ref{defhlbarnorm}) this immediately establishes the
first estimate in~(\ref{lemma:sobzeromass1}).

In order to complete the proof one only has to find a rigorous upper
bound on the second factor in the last line of the above estimate.
For this, one can first use the proof of~\cite[Corollary~3.3]{wanner:18b}
to establish the tail bound
\begin{displaymath}
  \sum_{k \in \N_0^d, \; |k| \ge N} c_k^2 \kappa_k^{-2}
  \; \le \;
  \frac{2^d}{\pi^4} \cdot \gamma_d(N) \; ,
\end{displaymath}
where~$\gamma_d(N)$ is explicitly defined in~\cite[Equation~(16)]{wanner:18b}.
This in turn yields the estimate
\begin{displaymath}
  \sum_{k \in \N_0^d, \; |k|>0} c_k^2 \kappa_k^{-2}
  \quad\le\quad
  \sum_{k \in \N_0^d, \; 0<|k| <  N} c_k^2 \kappa_k^{-2} \; + \;
    \frac{2^d}{\pi^4} \cdot \gamma_d(N) \; .
\end{displaymath}
Evaluating the finite sum and the tail bound using interval arithmetic
and $N = 1000$ then furnishes the constant in Table~\ref{table1}.
\end{proof}

The next lemma derives explicit bounds for the norm equivalence
of the norms on the Hilbert spaces~$\overline{\cH}^2$ and on~$\cH^2$,
which contain functions of zero and nonzero average, respectively.
\begin{lemma}[Norm equivalence between zero and nonzero mass]
\label{lemma:sobnobar2bar}
For all $u \in \overline{\cH}^2$ we have
\begin{displaymath}
  \| u \|_{\overline{\cH}^2} \le
  \| u \|_{\cH^2} \le
  \frac{\sqrt{1 + \pi^4}}{\pi^2} \| u \|_{\overline{\cH}^2} \;. 
\end{displaymath}
\end{lemma}
\begin{proof}
The first inequality is clear from the definitions of the two
norms in the last section, since $\kappa_k^2 \le 1 + \kappa_k^2$. 
For the second inequality, note that for $|k|>0$ one has the
inequality $\kappa_k = \pi^2 |k|^2 \ge \pi^2$, and therefore
\begin{displaymath}
  1+ \kappa_k^2 =
  \kappa_k^2 \left( 1  + \frac{1}{\kappa_k^2} \right) \le
  \kappa_k^2 \left(1 + \frac{1}{\pi^4} \right) =
  \kappa_k^2 \, \frac{1 + \pi^4}{\pi^4} \;. 
\end{displaymath}
This in turn implies
\begin{displaymath}
  \| u \|_{\cH^2}^2 =
  \sum_{k \in \N_0^d, \; |k|>0} (1 + \kappa_k^2) \alpha_k^2 \le
  \frac{1 + \pi^4}{\pi^4} \sum_{k \in \N_0^d, \; |k|>0}
    \kappa_k^2 \alpha_k^2 =
  \frac{1 + \pi^4}{\pi^4} \| u \|_{\overline{\cH}^2}^2 \; ,
\end{displaymath}
which completes the proof of the lemma.
\end{proof}
Note that from the above lemma one could conclude
$\overline{C}_m \le (\sqrt{1 + \pi^4}/\pi^2) C_m$, but the results
given in Lemma~\ref{lemma:sobzeromass} are around an order of
magnitude better. 

Our specific norm choice on the spaces~$\overline{\cH}^\ell$
has some convenient implications for its relation to the
Laplacian operator~$\Delta$. Clearly for any function
$u \in \overline{\cH}^\ell$ we have both $\Delta u \in
\overline{\cH}^{\ell-2}$ and $\Delta^{-1} u \in
\overline{\cH}^{\ell+2}$. Furthermore, if~$u$
is of the form
\begin{displaymath}
  u = \sum_{k \in \N_0^d, \; |k|>0} \alpha_k \phi_k \;,
  \qquad\mbox{ then }\qquad
  -\Delta u = \sum_{k \in \N_0^d, \; |k|>0}
    \kappa_k \alpha_k \phi_k \;,
\end{displaymath}
and we obtain the representation for $-\Delta^{-1}u$ if we
replace~$\kappa_k$ in the last sum by~$\kappa_k^{-1}$. This
immediately yields
\begin{eqnarray*}
  \| \Delta u \|_{\overline{\cH}^{\ell-2}}^2 &=&
    \sum_{k \in \N_0^d, \; |k|>0} \kappa_k^{\ell-2}
    \kappa_k^2 \alpha_k^2 \;, \\[1.5ex]
  \| u \|_{\overline{\cH}^\ell}^2 &=&
    \sum_{k \in \N_0^d, \; |k|>0} \kappa_k^\ell
    \alpha_k^2 \;, \\[1.5ex]
  \| \Delta^{-1} u \|_{\overline{\cH}^{\ell+2}}^2 &=&
    \sum_{k \in \N_0^d, \; |k|>0} \kappa_k^{\ell+2}
      \kappa_k^{-2} \alpha_k^2 \; ,
\end{eqnarray*}
and altogether we have verified the following lemma.
\begin{lemma}[The Laplacian is an isometry] \label{lemma:lapiso}
For every $\ell \in \Z$ the Laplacian operator~$\Delta$ is an isometry
from~$\overline{\cH}^\ell$ to~$\overline{\cH}^{\ell-2}$, i.e., we have
\begin{eqnarray*}
  \| \Delta^{-1} u \|_{\overline{\cH}^{\ell+2}} =
  \| u \|_{\overline{\cH}^\ell} =
  \| \Delta u \|_{\overline{\cH}^{\ell-2}} \; .
\end{eqnarray*}
\end{lemma}
To close this section we present a final result which relates the
standard norm in the Hilbert space~$\overline{\cH}^\ell$ to the norm
in~$\overline{\cH}^m$ if~$\ell \le m$. This inequality will turn out
to be useful later on.
\begin{lemma}[Relating the norms in $\overline{\cH}^\ell$ and $\overline{\cH}^m$]
\label{lemma:bscaleest}
For all $u \in \overline{\cH}^m$ and all $\ell \le m$ we have the estimate
\begin{displaymath}
  \| u \|_{\overline{\cH}^\ell} \; \le \;
  \frac{1}{\pi^{m-\ell}} \, \| u \|_{\overline{\cH}^m}\;. 
\end{displaymath}
Furthermore, note that in the special case $\ell = 0 \le m$ we have
$\| u \|_{\overline{\cH}^0} = \| u \|_{L^2}$.
\end{lemma}
\begin{proof}
Suppose that $u \in \overline{\cH}^m$ is given by
$u = \sum_{k \in \N_0^d, \; |k|>0} \alpha_k \phi_k$.
Then we have
\begin{displaymath}
  \| u \|_{\overline{\cH}^\ell}^2 \; = \; \sum_{k \in \N_0^d, \; |k|>0}
    \frac{\kappa_k^m \alpha_k^2}{\kappa_k^{m-\ell}} \; \le \;
  \frac{1}{\pi^{2(m-\ell)}} \sum_{k \in \N_0^d, \; |k|>0}
    \kappa_k^m \alpha_k^2 \; = \;
  \frac{1}{\pi^{2(m-\ell)}} \| u \|_{\overline{\cH}^m}^2 \; ,
\end{displaymath}
since for all $|k| > 0$ one has $\kappa_k \ge \pi^2$.
\end{proof}
\subsection{Projection operators}
\label{sec:proj}
In order to establish computer-assisted existence proofs for
equilibrium solutions of~(\ref{eqn:dbcp}) one needs to work
with suitable finite-dimensional approximations. In our framework,
we use truncated cosine series, and this is formalized in the current
section through the introduction of suitable projection operators.

For this, let~$N \in \N$ denote a positive integer, and
consider $u \in \cH^\ell$ for $\ell \in \N_0$, or alternatively
$u \in \overline{\cH}^\ell$ for $\ell \in \Z$, of the form
$u = \sum_{k \in \N_0^d} \alpha_k \phi_k$, where in the latter
case $\alpha_0 = 0$. Then we define the projection
\begin{equation} \label{eqn:defpn}
  P_N u = \sum_{k \in \N_0^d, \; |k|_\infty< N} \alpha_k \phi_k \; . 
\end{equation}
Note that in this definition we use the $\infty$-norm of the
multi-index~$k$, since this simplifies the implementation of our
method. The so-defined operator~$P_N$ is a bounded linear operator
on~$\cH^\ell$ with induced operator norm~$\| P_N \| = 1$, and one
can easily see that it leaves the space~$\overline{\cH}^\ell$
invariant if $\ell \in \Z$. Furthermore, it is straightforward to
show that for any $N \in \N$ we have
\begin{displaymath}
  \dim P_N {\cH}^\ell = N^d
  \qquad\mbox{ and }\qquad
  \dim P_N \overline{\cH}^\ell = N^d - 1 \;. 
\end{displaymath}
For all $\ell \in \N_0$ we would like to point out that 
$(I - P_1) {\cH}^\ell = \overline{\cH}^\ell$. Since this is
an especially useful operator, we introduce the abbreviation
\begin{equation}
  \overline{P} = I - P_1 \; . 
\end{equation}
The operator~$\overline{P}$ satisfies the following useful
identity.
\begin{lemma}
\label{lemma:pbarl2prod}
For arbitrary $u \in \cH^0$ and $v \in \overline{\cH}^0$ we
have the equality
\begin{displaymath}
  \left( \overline{P} u, v \right)_{L^2} =
  (u,v)_{L^2} \; .
\end{displaymath}
\end{lemma}
\begin{proof}
This result can be established via direct calculation. Note that
\begin{eqnarray*}
  \left( \overline{P} u, v \right)_{L^2} &=&
    (u - \alpha_0 \phi_0, v)_{L^2} =
    (u,v)_{L^2} - \alpha_0 (\phi_0,v)_{L^2} \\
    &=& (u,v)_{L^2} - \alpha_0 \int_\Omega v(x) \; dx =
    (u,v)_{L^2} - 0 \; ,
\end{eqnarray*}
where for the last step we used the fact that $v \in \overline{\cH}^0$.
\end{proof}
We close this section by deriving a norm bound for the 
infinite cosine series part that is discarded by the 
projection~$P_N$ in terms of a higher-regularity norm.
More precisely, we have the following.
\begin{lemma}[Projection tail estimates]
\label{lemma:projtailest}
Consider two integers $\ell \le m$ and let the function $u \in \overline{\cH}^m$
be arbitrary. Then the projection tail~$(I - P_N)u$ satisfies
\begin{displaymath}
  \| (I - P_N) u \|_{\overline{\cH}^\ell} \; \le \;
  \frac{1}{\pi^{m - \ell} N^{m - \ell}} \,
    \| (I - P_N) u \|_{\overline{\cH}^m} \; \le \;
  \frac{1}{\pi^{m - \ell} N^{m - \ell}} \,
    \| u \|_{\overline{\cH}^m} \;. 
\end{displaymath}
\end{lemma}
\begin{proof}
Suppose that $u \in \overline{\cH}^m$ is given by
$u = \sum_{k \in \N_0^d, \; |k|>0} \alpha_k \phi_k$.
Then we have
\begin{eqnarray*}
  \| (I-P_N) u \|_{\overline{\cH}^\ell}^2 & = &
    \sum_{k \in \N_0^d, \;|k|_\infty \ge N} \kappa_k^\ell \alpha_k^2 \; = \;
    \sum_{k \in \N_0^d, \;|k|_\infty \ge N} \frac{\kappa_k^m \alpha_k^2}
    {\kappa_k^{m-\ell}} \\[1.5ex]
  & \le & \sum_{k \in \N_0^d, \;|k|_\infty \ge N}
    \frac{\kappa_k^m \alpha_k^2}{(\pi^2 N^2)^{m-\ell}} \; = \;
    \frac{1}{(\pi^2 N^2)^{m-\ell}} \;
    \| (I-P_N) u \|_{\overline{\cH}^m}^2 \; , 
\end{eqnarray*}
since the estimate $|k|_\infty \ge N$ yields $|k| \ge N$.
\end{proof}
\section{Derivative inverse estimate}\label{subsec:die}
This section is devoted to establishing derivative inverse bound in 
hypothesis~(H2), which is required for Theorem~\ref{nift:thm}, the
constructive implicit function theorem. More precisely, our goal in
the following is to derive a constant~$K$ such that 
\begin{displaymath}
  \left\| (D_u F)^{-1} \right\|_{\cL(Y,X)} \le K \; ,
\end{displaymath}
i.e., we need to find a bound on the operator norm of the inverse
of the Fr\'echet derivative of~$F$ with respect to~$u$. We divide 
the derivation of this estimate into four parts. In
Section~\ref{sec:investoutline} we give an outline of our approach,
introduce necessary definitions and auxiliary results, and present
the main result of this section. This result will be verified in the
following three sections. First, we discuss the finite-dimensional
projection of~$D_u F$ in Section~\ref{sec:findim}. Using this
finite-dimensional operator, we then construct an approximative
inverse to the Fr\'echet derivative in Section~\ref{sec:approxinverse},
before everything is assembled to provide the desired estimate in
the final Section~\ref{sec:comprof}.
\subsection{General outline and auxiliary results}
\label{sec:investoutline}
For convenience of notation in the subsequent discussion, for
fixed parameters and~$u$ we abbreviate the Fr\'echet derivative
of~$F$ by
\begin{equation}  \label{eqn:ldef}
  L v = D_u F(\lambda, \sigma, \mu, u) [v] \; , \quad
  L \in \cL( X, Y ) \; , \quad\mbox{ with }\quad
  X = \overline{\cH}^{2} \; , \quad
  Y = \overline{\cH}^{-2} \; .
\end{equation}
Standard results imply that~$L$ is a bounded linear
operator $L \in \cL(\overline{\cH}^{2},\overline{\cH}^{-2})$, which
explicitly is given by
\begin{equation}\label{eqn:ldefinition}
  L v = -\Delta ( \Delta v + \lambda \, f'(u + \mu) v ) -
    \lambda \sigma v \; . 
\end{equation}
More precisely, note that since the nonlinearity~$f$ is twice continuously differentiable,
and in view of Sobolev's imbedding recalled in~(\ref{eqn:cmcb}), the
function~$f'(u + \mu)$ is continuous on~$\overline{\Omega}$, which
makes the product $\lambda f'(u + \mu) v$ an $L^2(\Omega)$-function,
and therefore $- \Delta(\lambda f'(u + \mu) v) \in \overline{\cH}^{-2}$.
We will also use the abbreviation
\begin{equation} \label{eqn:defq}
  q(x) = \lambda f'(u(x) + \mu) \;. 
\end{equation}
As mentioned earlier, the constructive implicit function theorem
crucially relies on being able to find a bound~$K$ such that
$\|L^{-1}\| \le K$. Our goal is to do so by using a finite-dimensional
approximation for~$L$, since that can be analyzed via rigorous
computational means. Our finite-dimensional approximation for~$L$
is given as follows. For fixed $N \in \N$ define the finite-dimensional
spaces
\begin{displaymath}
  X_N = P_N X
  \qquad\mbox{ and }\qquad
  Y_N = P_N Y \; ,
\end{displaymath}
where the projection operator is given in~(\ref{eqn:defpn}).
Define $L_N: X_N \to Y_N$ by 
\begin{equation} \label{eqn:defln}
  L_N = \left. P_N L \right|_{X_N} \; .
\end{equation}
Let~$K_N$ be a bound on the inverse of the finite-dimensional
operator~$L_N$, i.e., suppose that
\begin{equation} \label{eqn:defkn}
  \left\| L_N^{-1} \right\|_{\cL(Y_N,X_N)} \le K_N \; ,
\end{equation}
where the spaces~$X_N$ and~$Y_N$ are equipped with the norms
of~$X$ and~$Y$, respectively. We will discuss further details
on appropriate coordinate systems and the actual computation
of both~$L_N$ and~$K_N$ in Section~\ref{sec:findim}. Our main
result for this section is as follows.
\begin{theorem}[Derivative inverse estimate]
\label{thm:k}
Assume there is a constant~$\tau > 0$ and an integer $N \in \N$ such that
\begin{displaymath}
  \frac{1}{\pi^2 N^2} \sqrt{ K_N^2  \, \| q \|_\infty^2 +
    C_b^2 \, \frac{1+\pi^4}{\pi^4} \, \| q \|_{\cH^2}^2}
  \; \le \; \tau \; < \; 1 \; ,
\end{displaymath}
where~$K_N$ and~$q$ are defined in~(\ref{eqn:defkn})
and~(\ref{eqn:defq}), respectively. Then the derivative
operator~$L$ in~(\ref{eqn:ldefinition}) satisfies
\begin{displaymath}
  \left\| L^{-1} \right\|_{\cL(X,Y)} \le
  \frac{\max ( K_N, 1) }{1-\tau} \; . 
\end{displaymath}
\end{theorem}
Before we begin to prove this main theorem, we state a
necessary result which is based on a Neumann series argument
to derive bounds on the operator norm of an inverse of an
operator. This is a standard functional-analytic technique,
which we state here for the reader's convenience. A proof
can be found in~\cite[Lemma~4]{sander:wanner:16a}.
\begin{proposition}[Neumann series inverse estimate]
\label{prop:neumann}
Let $\cA \in \cL(X,Y)$ be an arbitrary bounded linear operator
between two Banach spaces, and let $\cB \in \cL(Y,X)$ be one-to-one.
Assume that there exist positive constants~$\rho_1$ and~$\rho_2$
such that
\begin{displaymath}
  \| I - \cB \cA \|_{\cL(X,X)} \le \rho_1 < 1
  \qquad\mbox{ and }\qquad
  \|\cB\|_{\cL(Y,X)} \le \rho_2 \;. 
\end{displaymath}
Then $\cA$ is one-to-one and onto, and 
\begin{displaymath}
\| \cA^{-1}\|_{\cL(Y,X)} \le \frac{\rho_2}{1-\rho_1} \;.
\end{displaymath}
In subsequent discussions, we will refer to~$\cB$ as an
{\em approximate inverse}.
\end{proposition}
We are now ready to proceed with the proof of the main result
of the section, Theorem~\ref{thm:k}. For this, we fix all
parameters, as well as $u \in \overline{\cH}^2$. Our goal is
to prove that~$L$ is one-to-one, onto, and has an inverse whose
operator norm is bounded by the value $K = \max(K_N,1)/(1-\tau)$.
\subsection{Finite-dimensional projections of the linearization}
\label{sec:findim}
In this section, we consider~$L_N$, the finite dimensional projection
of the operator~$L$. The linear map~$L_N$ is tractable using rigorous
computational methods, since calculating a finite-dimensional inverse
is something that can be done using numerical linear algebra. To
derive~$L_N$ in more detail, we recall the definitions of the
following projection spaces, all of which are Hilbert spaces:
\begin{displaymath}
  \begin{array}{rclcrclcrcl}
    X & = & \overline{\cH}^2 \;, & \quad &
      X_N & = & P_N X \; , & \quad &
      X_\infty & = & (I- P_N) X \; , \\[1ex]
    Y & = & \overline{\cH}^{-2} \;, & \quad &
      Y_N & = & P_N Y \; , & \quad &
      Y_\infty & = & (I- P_N) Y \;.
  \end{array}
\end{displaymath}
Recall that in~(\ref{eqn:defln}) we defined $L_N: X_N \to Y_N$ via
$L_N = \left. P_N L \right|_{X_N}$. In order to work with this operator
in a straightforward computational manner, we need to find its matrix
representation. Since both~$X_N$ and~$Y_N$ have the basis~$\phi_k$
for all $k \in \N_0^d$ with~$0 < |k|_\infty < N$, one obtains such
a matrix~$B = (b_{k,\ell}) \in \R^{(N^d-1) \times
(N^d-1)}$ via the definition
\begin{displaymath}
  b_{k,\ell} = (L \phi_\ell,\phi_k)_{L^2} =
  (L_N \phi_\ell,\phi_k)_{L^2} \; ,
\end{displaymath}
where~$k,\ell \in \N_0^d$ satisfy $0 < |k|_\infty < N$ and
$0 < |\ell|_\infty < N$.

The above matrix representation characterizes~$L_N$ on the
algebraic level in the following sense. If we consider a function
$v_N \in X_N$, introduce the representations
\begin{displaymath}
  v_N  = \sum_{k \in \N_0^d, \, 0 < |k|_\infty < N}
    \alpha_k \phi_k(x)
  \qquad\mbox{ and }\qquad
  L_N v_N  = \sum_{k \in \N_0^d, \, 0 < |k|_\infty < N}
    \beta_k \phi_k(x) \; ,
\end{displaymath}
and if we collect the numbers~$\alpha_k$ and~$\beta_k$ in
vectors~$\alpha$ and~$\beta$ in the straightforward way, then
we have
\begin{displaymath}
  \beta = B \alpha \; .
\end{displaymath}
This natural algebraic representation has one drawback. We would like
to use the regular Euclidean norm on real vector spaces, as well as
the induced matrix norm, to study the $\cL(X_N,Y_N)$-norm of~$L_N$.
To achieve this, we recall Lemma~\ref{lemma:orthoghell} which shows that
the collection~$\{ \kappa_k^{-1} \phi_k(x) \}$ with~$k$ as above is
an orthonormal basis in $X_N \subset X$, and~$\{ \kappa_k \phi_k(x) \}$
is an orthonormal basis in $Y_N \subset Y$. Thus, we need to use the
representations
\begin{displaymath}
  v_N  = \sum_{k \in \N_0^d, \, 0 < |k|_\infty < N}
    \tilde{\alpha}_k \kappa_k^{-1} \phi_k(x)
  \qquad\mbox{ and }\qquad
  L_N v_N  = \sum_{k \in \N_0^d, \, 0 < |k|_\infty < N}
    \tilde{\beta}_k \kappa_k \phi_k(x)
\end{displaymath}
instead of the ones given above. In order to pass back and forth 
between these two representations we define the diagonal matrix
\begin{displaymath}
  D = \left( \begin{array}{cccc}
      \kappa_1 & 0 & \cdots & 0 \\
      0 & \kappa_2 & \ddots & \vdots \\
      \vdots & \ddots & \ddots & 0 \\
      0 & \cdots & 0 & \kappa_{N-1}
    \end{array} \right) \; .
\end{displaymath}
One can easily see that on the level of vectors we have
\begin{displaymath}
  \alpha = D^{-1} \tilde{\alpha}
  \quad\mbox{ and }\quad
  \beta = D \tilde{\beta} \; ,
  \quad\mbox{ and therefore }\quad
  \tilde{\beta} = D^{-1} B D^{-1} \tilde{\alpha} \; .
\end{displaymath}
In view of Lemma~\ref{lemma:orthoghell} one then obtains
\begin{displaymath}
  \left\| L_N \right\|_{\cL(X_N,Y_N)} =
  \| \tilde{B} \|_2
  \qquad\mbox{ with }\qquad
  \tilde{B} = D^{-1} B D^{-1} \; ,
\end{displaymath}
where~$\| \cdot \|_2$ denotes the regular induced $2$-norm of a matrix.
Moreover, one can verify that we also have the identity
\begin{equation} \label{eqn:lninversenorm}
  \left\| L_N^{-1} \right\|_{\cL(Y_N,X_N)} =
  \left\| \tilde{B}^{-1} \right\|_{L^2} \; .
\end{equation}
In other words, using this formula, we can use interval arithmetic to
establish a rigorous upper bound on the norm of this finite-dimensional
inverse.

So far our considerations applied to any bounded linear operator between
the spaces~$X$ and~$Y$. Specifically for the linearization of the diblock
copolymer equation we can derive an explicit formula for the matrix 
entries~$b_{k,\ell}$. Recall that~$\phi_k$ as defined in~(\ref{eqn:phik})
is an eigenfunction for the negative Laplacian~$-\Delta$ with
eigenvalue~$\kappa_k$. Therefore, for all multi-indices $k,\ell \in \N_0^d$
with $0 < |k|_\infty < N$ and $0 < |\ell|_\infty < N$ one obtains
\begin{eqnarray}
  b_{k,\ell} & = & (L \phi_\ell,\phi_k)_{L^2} \; = \;
    (-\kappa_k^2 - \lambda \sigma) (\phi_k,\phi_\ell)_{L^2} -
    (\Delta(\lambda f'(u+\mu) \phi_\ell),\phi_k)_{L^2}
    \nonumber \\[1ex]
  & = & (-\kappa_k^2 - \lambda \sigma) \delta_{k,\ell} -
    (\Delta(q \phi_\ell),\phi_k)_{L^2} \nonumber \\[1ex]
  & = & (-\kappa_k^2 - \lambda \sigma) \delta_{k,\ell} -
    (q \phi_\ell,\Delta \phi_k)_{L^2} \nonumber \\[1ex]
  & = & -\left( \kappa_k^2 + \lambda \sigma \right) \delta_{k,\ell} +
    \kappa_k \left( q \phi_\ell, \phi_k \right)_{L^2}  \;.
    \label{eqn:defbkell}
\end{eqnarray}
The above formula explicitly gives the entries of the matrix~$B$.
For our computer-assisted proof, we are however interested in the
scaled matrix~$\tilde{B} = D^{-1} B D^{-1}$. One can immediately
verify that its entries~$\tilde{b}_{k,\ell}$ are given by
\begin{equation} \label{eqn:deftildebkell}
  \tilde{b}_{k,\ell} \; = \;
  -\left(1 + \frac{\lambda \sigma}{\kappa_k^2} \right) \delta_{k,\ell} +
    \frac{1}{\kappa_\ell} (q \phi_\ell, \phi_k)_{L^2}
  \quad\mbox{ with }\quad
  q(x) = \lambda f'(\mu + u(x)) \; .
\end{equation}
In view of~(\ref{eqn:lninversenorm}), this formula will allow us
to bound the operator norm of the inverse of the finite-dimensional
projection~$L_N$ using techniques from interval arithmetic.
\subsection{Construction of an approximative inverse}
\label{sec:approxinverse}
The crucial part in the derivation of our norm bound for the
inverse of~$L$ is the application of Proposition~\ref{prop:neumann}.
For this, we need to construct an approximative inverse of this
operator. Since this construction has to be explicit, we will 
approach it in two steps. The first has already been accomplished
in the last section, where we considered a finite-dimensional
projection of~$L$, which can easily be inverted numerically. 
In this section, we complement this finite-dimensional part with
a consideration of the infinite-dimensional complementary space.
For this, we refer the reader again to the definition of the
matrix representation~$B$ in~(\ref{eqn:defbkell}). As $N \to \infty$,
this representation leads to better and better approximations of
the operator~$L$. Note in particular that the entry~$b_{k,\ell}$
is the sum of two terms. The first of these is a diagonal matrix,
and its entries clearly dominate the second term in~(\ref{eqn:defbkell}).
We therefore use the inverse of the first term in order to complement
the inverse of~$L_N$.

To describe this procedure in more detail, suppose that the function
$v \in Y$ is given by
\begin{displaymath}
  v = \sum_{k \in \N_0^d, \, |k|_\infty > 0}
    \alpha_k \phi_k(x)
  = v_N + v_{\infty} \in
  Y_N \oplus Y_\infty \; ,
\end{displaymath}
where we define
\begin{displaymath}
  Y_N = P_N Y
  \qquad\mbox{ and }\qquad
  Y_\infty = \left( I - P_N \right) Y \; .
\end{displaymath}
Using this representation the approximative inverse~$S \in \cL(Y,X)$
of~$L \in \cL(X,Y)$ is defined via the formula
\begin{displaymath}
  S v = L_N^{-1} v_N -
    \sum_{k \in \N_0^d, \, |k|_\infty \ge N} \frac{\alpha_k}
      {\kappa_k^2 + \lambda \sigma} \, \phi_k \; . 
\end{displaymath}
In addition, consider the operator $T = S|_{Y_\infty}$, i.e., let
\begin{displaymath}
  T \sum_{k \in \N_0^d, \, |k|_\infty\ge N} \alpha_k \phi_k =
  -\sum_{k \in \N_0^d, \, |k|_\infty\ge N} \frac{\alpha_k}
    {\kappa_k^2 + \lambda \sigma} \phi_k \; .
\end{displaymath}
One can easily see that~$T : Y_\infty \to X_\infty = (I-P_N)X$
is one-to-one and onto, and in fact we have the identity
\begin{displaymath}
  T^{-1} \sum_{k \in \N_0^d, \, |k|_\infty\ge N} \alpha_k \phi_k =
  -\sum_{k \in \N_0^d, \, |k|_\infty\ge N} \left( \kappa_k^2 +
    \lambda \sigma \right) \alpha_k \phi_k \; ,
\end{displaymath}
which can be rewritten in the form
\begin{equation} \label{eqn:tinv}
  T^{-1} v_\infty = -\left( \Delta^2 v_\infty  +
    \lambda \sigma v_\infty \right) \; . 
\end{equation}
Also, from the definition of~$S$ we get the
alternative representation
\begin{equation} \label{eqn:approxinv}
  Sv = L_N^{-1} v_N + T v_\infty \; .
\end{equation}
To close this section, we now derive a bound on the 
operator norm of~$S$, since this will be needed in the
application of Proposition~\ref{prop:neumann}. As a first
step, we show that $\| T v_\infty \|_X \le \| v_\infty \|_Y$
for all $y_\infty \in Y_\infty$, which follows readily from
\begin{eqnarray*}
  \left\| T \sum_{k \in \N_0^d, \, |k|_\infty\ge N}
    \alpha_k \phi_k \right\|_{X}^2 & = &
    \left\| \sum_{k \in \N_0^d, \, |k|_\infty\ge N}
    \frac{\alpha_k}{\kappa_k^2 + \lambda \sigma} \phi_k
    \right\|_{\overline{\cH}^2}^2 \\[2ex]
  & = &
    \sum_{k \in \N_0^d, \, |k|_\infty\ge N}
    \frac{\alpha_k^2 \kappa_k^2}{(\kappa_k^2 +
    \lambda \sigma)^2} \\[2ex]
  & \le &
    \sum_{k \in \N_0^d, \, |k|_\infty\ge N}
    \frac{\alpha_k^2 \kappa_k^2}{(\kappa_k^2)^2}
  \; = \;
    \sum_{k \in \N_0^d, \, |k|_\infty\ge N}
    \kappa_k^{-2} \alpha_k^2 \\[2ex]
  & = &
    \left\| \sum_{k \in \N_0^d, \, |k|_\infty\ge N}
    \alpha_k \phi_k \right\|_{\overline{\cH}^{-2}}^2
  \; = \;
    \left\| \sum_{k \in \N_0^d, \, |k|_\infty\ge N}
    \alpha_k \phi_k \right\|_Y^2.
\end{eqnarray*}
This estimate in turn implies for all $v = v_N + v_\infty
\in Y_N \oplus Y_\infty$ the estimate
\begin{eqnarray*}
  \| S v \|_X^2 & = &
    \| L_N^{-1} v_N \|_X^2 + \| T v_\infty \|_X^2 \\[2ex]
  & \le &
    \underbrace{\| L_N^{-1} \|_{\cL(Y_N,X_N)}^2}_{\le K_N^2}
    \| v_N \|_Y^2 + \| v_\infty \|_Y^2
    \; \le \; \max(K_N,1)^2 \| v \|_Y^2 \; ,
\end{eqnarray*}
where we used the definition of~$K_N$ from~(\ref{eqn:defkn}).
Altogether, we have shown that
\begin{equation} \label{eqn:sbound}
  \|S\|_{\cL(Y,X)} \le \max(K_N,1) \; . 
\end{equation}
In other words, the operator norm of the approximate inverse~$S$
given in~(\ref{eqn:approxinv}) can be bounded in terms of the
inverse bound for the finite-dimensional projection given
in~(\ref{eqn:defkn}). Furthermore, it follows directly from 
the definition of~$S$ that this operator is one-to-one.
\subsection{Assembling the final inverse estimate}
\label{sec:comprof}
In the last section we addressed two crucial aspects of
Proposition~\ref{prop:neumann}. On the one hand, we provided an
explicit construction for the approximative inverse~$S \in \cL(Y,X)$
of the Fr\'echet derivative~$L$ defined in~(\ref{eqn:ldef}). On the
other hand, we derived an upper bound on the operator norm of~$S$,
which can be computed using the finite-dimensional projection~$L_N$
of~$L$. This in turn provides the constant~$\rho_2$ in
Proposition~\ref{prop:neumann}. In this final subsection, we focus
on the constant~$\rho_1$, i.e., we derive an upper bound on the
norm~$\|I - S L \|_{\cL(X,X)}$, and show how this bound can be made
smaller than one. Altogether, this will complete the proof of the
estimate for the constant~$K$ in the constructive implicit function
theorem, which was given in Theorem~\ref{thm:k}.

Before we begin, recall the abbreviation $q(x) = \lambda f'(u(x) + \mu)$.
From our definitions of the operators~$L \in \cL(X,Y)$, $S \in \cL(Y,X)$,
$L_N \in \cL(X_N,Y_N)$, and~$T \in \cL(Y_\infty,X_\infty)$, as well
as the projection~$P_N$, and using the additive representation $v = v_N + v_\infty
\in Y_N \oplus Y_\infty$, we have the identity
\begin{equation}\label{eqn:lv}
  Lv = \left( L_N v_N - P_N \Delta(q v_\infty) \right) +
       \left( T^{-1} v_\infty - \left( I-P_N \right)
       \Delta(q v) \right) \; ,
\end{equation}
which will be derived in detail in the following calculation. Notice
that the first parentheses contain only terms in the finite-dimensional
space~$Y_N$, while the second parentheses contain terms in~$Y_\infty$.
With this in mind, we have
\begin{eqnarray*}
  Lv & = &
    -\Delta\left( \Delta v + q v \right) - \lambda \sigma v \\[1ex]
  & = & -\Delta^2 v_N - \Delta^2 v_\infty - P_N \Delta (q v_N) -
    (I - P_N) \Delta (q v_N) \\[0.5ex]
    & & \qquad - \Delta (q v_\infty) -
    \lambda \sigma v_N - \lambda \sigma v_\infty \\[1ex]
  & = & \left( -\Delta^2 v_N - P_N \Delta (q v_N) -
    \lambda \sigma v_N \right) - \left( \Delta^2 v_\infty +
    \lambda \sigma v_\infty \right) \\[0.5ex]
    & & \qquad - (I - P_N) \Delta (q v_N) -
    \Delta (q v_\infty) \\[1ex]
  & = & L_N v_N + T^{-1} v_\infty - (I - P_N) \Delta (q v_N) -
    P_N \Delta (q v_\infty) - (I - P_N) \Delta (q v_\infty) \\[1ex]
  & = & L_N v_N + T^{-1} v_\infty - P_N \Delta (q v_\infty) -
    (I - P_N) \Delta (q v) \; .
\end{eqnarray*}
The first two lines follow just from the definitions, projections, and
rearrangements of terms. The third line is a consequence of~(\ref{eqn:lv})
and~(\ref{eqn:tinv}). Finally, the fourth and fifth lines involve only
rearrangements using the projection operator.

Using the above representation~(\ref{eqn:lv}) of the operator~$L$ which
is split along the subspaces~$Y_N$ and~$Y_\infty$, we can now derive an
expression for~$I - SL \in \cL(X,X)$. More precisely, we
have
\begin{equation}\label{eqn:imsl}
  (I - SL)v = L_N^{-1} P_N \Delta(q v_\infty) +
              T (I - P_N) \Delta (q v) \; ,  
\end{equation}
and this will be verified in detail below. Notice that in this
representation, the first term of the right-hand side lies in the
finite-dimensional space~$X_N$, while the second term is contained
in the complement~$X_\infty$. The identity in~(\ref{eqn:imsl}) now
follows from~(\ref{eqn:approxinv}) and
\begin{eqnarray*}
  SL v & = & L_N^{-1} \left( L_N v_N - P_N \Delta(q v_\infty) \right) +
    T \left( T^{-1} v_\infty - (I - P_N) \Delta(qv) \right) \\[1ex]
  & = & v_N - L_N^{-1} P_N \Delta(q v_\infty) +
    v_\infty - T (I - P_N) \Delta(qv) \\[1ex]
  & = & I v  - L_N^{-1} P_N \Delta(q v_\infty) -
    T (I - P_N) \Delta(qv) \; .
\end{eqnarray*}
After these preparation, we can now show that the operator norm
of~$I - SL$ can be expected to be small for sufficiently large $N$. This will provide an estimate
for the constant~$\rho_1$ in Proposition~\ref{prop:neumann}, and
conclude the proof of Theorem~\ref{thm:k}. In order to show that
$\|I - SL \|_{\cL(X,X)}$ is indeed small, we separately bound the
two terms in~(\ref{eqn:imsl}) as
\begin{displaymath}
  \begin{array}{rclcrcl}
    \DS \left\| L_N^{-1} P_N \Delta(q v_\infty) \right\|_X
      & \le & \DS A \|v\|_X & \quad\mbox{ with }\quad &
      \DS A & := & \DS \frac{K_N \|q\|_\infty}{\pi^2 N^2}
      \; , \\[3ex]
    \DS \left\| T (I-P_N) \Delta (q v) \right\|_X & \le &
      \DS B \|v\|_X & \quad\mbox{ with }\quad &
      \DS B & := & \DS \frac{C_b \sqrt{1+\pi^4} \,
      \|q\|_{\cH^2}}{\pi^4 N^2} \; .
  \end{array}
\end{displaymath}
The first of these inequalities is established in the following
calculation, which makes liberal use of Sobolev embeddings and
other established inequalities:
\begin{eqnarray*}
  \left\| L_N^{-1} P_N \Delta(q v_\infty) \right\|_X & \le &
    \left\| L_N^{-1} \right\|_{\cL(Y_N,X_N)} \|P_N
    \Delta(q v_\infty) \|_Y \\[1.5ex]
  & \le & K_N \left\| P_N \Delta(q v_\infty)
    \right\|_{\overline{\cH}^{-2}}
    \; \le \; K_N \left\| \Delta(q v_\infty)
    \right\|_{\overline{\cH}^{-2}} \\[1.5ex]
  & \le &
    K_N \| q v_\infty \|_{{\cH}^{0}}
    \; \le \;
    K_N \|q\|_{\infty} \left\| (I-P_N) v
    \right\|_{\overline{\cH}^0} \\[1.5ex]
  & \le & K_N \|q\|_{\infty} \,
    \frac{\| v \|_{\overline{\cH}^{2}}}{\pi^2 N^2}
    \; = \;
    \frac{K_N \|q\|_\infty}{\pi^2 N^2} \, \|v\|_X
    \; = \;
    A \|v\|_X \; ,
\end{eqnarray*}
where for the last inequality we used Lemma~\ref{lemma:projtailest}.
The second estimate, the one involving the constant~$B$, is verified
as follows, again with help from our previously derived inequalities,
in particular the fact that~$\| T \|_{\cL(Y_\infty,X_\infty)} \le 1$
and Lemmas~\ref{lemma:sobnobar2bar} and~\ref{lemma:projtailest}:
\begin{eqnarray*}
  \left\| T (I-P_N) \Delta (q v) \right\|_X & \le &
    \left\| (I-P_N) \Delta (q v) \right\|_{\overline{\cH}^{-2}}
    \; \le \;
    \frac{\| \Delta (q v)\|_{\overline{\cH}^{0}} }{\pi^2 N^2} \\[1.5ex]
  & = &
    \frac{\left\| \overline{P} (q v)
    \right\|_{\overline{\cH}^{2}}}{\pi^2 N^2}
    \; \le \;
    \frac{\| q  v \|_{\cH^{2}} }{\pi^2 N^2}
    \; \le \;
    \frac{C_b \| q \|_{\cH^{2}} \| v \|_{\cH^{2}} }{\pi^2 N^2} \\[1.5ex]
  & \le & 
    \frac{C_b \| q\|_{\cH^{2}}}{\pi^2 N^2} \cdot
    \frac{\sqrt{1+\pi^4}}{\pi^2} \cdot \|v \|_{\overline{\cH}^{2}}
    \; = \; B \|v\|_X \; .
\end{eqnarray*}
Now that we have established these two inequalities, the proof of
Theorem~\ref{thm:k} can easily be completed using an application of
Proposition~\ref{prop:neumann}. Specifically, the inequalities which
involve the constants~$A$ ands~$B$ combined with~(\ref{eqn:imsl})
imply that
\begin{displaymath}
  \| I  - S L \|_{\cL(X,X)} \; \le \;
  \sqrt{A^2 + B^2} \; = \;
  \frac{1}{\pi^2 N^2} \, \sqrt{K_N^2 \|q\|_\infty^2 +
    C_b^2 \, \frac{1+\pi^4}{\pi^4} \, \|q\|_{\cH^2}^2} \; .
\end{displaymath}
We also know from~(\ref{eqn:sbound}) that $\| S \|_X \le \max(K_N,1)$.
Therefore, we can directly apply Proposition~\ref{prop:neumann} with
the constants $\rho_1 = \sqrt{A^2 + B^2} \le \tau < 1$ and
$\rho_2 = \max(K_N,1)$, and this immediately implies that the
operator~$L \in \cL(X,Y)$ is one-to-one, onto, and the norm of its
inverse operator is bounded via
\begin{displaymath}
  \left\| L^{-1} \right\|_{\cL(Y,X)} \; \le \;
  \frac{\rho_2}{1-\rho_1} \; = \;
  \frac{\max(K_N,1)}{1-\tau} \; .
\end{displaymath}
This completes the proof of Theorem~\ref{thm:k}.
\section{Lipschitz estimates}
\label{sec:cift}
In this section, our goal is to establish the Lipschitz constants needed in 
hypotheses~(H3) and~(H4) required for Theorem~\ref{nift:thm}, the constructive
implicit function theorem. Namely, we need to establish Lipschitz bounds
for the derivatives of~$F$ 
with respect to both~$u$ and with respect to the continuation
parameter. We are considering single-parameter continuation, meaning that
we have three separate situations to discuss, corresponding to the three
different parameters~$\lambda$, $\sigma$, and~$\mu$. Specifically, for~$p$
being one of these three parameters, for a fixed parameter-function
pair~$(p^*,u^*) \in \R \times X$, and for fixed values of~$d_p$ and~$d_u$,
we assume that $|p - p^*| \le d_p$, and $\| u - u^* \|_X \le d_u$. Furthermore,
by a slight abuse of notation we drop the parameters different from~$p$
from the argument list of~$F$ in~(\ref{eqn:deffoperator}). Our goal in
the current section is to obtain tight and easily computable bounds on
the constants~$M_1$ through~$M_4$ in the following two formulas:  
\begin{equation} \label{eqn:lipschitz}
  \begin{array}{rcl}
  \DS \| D_u F ( p, u) - D_u F ( p, u) \| _{\cL(X,Y)} & \le &
    \DS M_1 \;  \| u - u^* \|_X + M_2 \; |p - p^*| \; , \\[1ex]
  \DS \| D_p F ( p, u) - D_p F ( p, u) \| _{\cL(\R,Y)} & \le & 
    \DS M_3 \;  \| u - u^* \|_X + M_4 \; |p - p^*| \; . 
  \end{array}
\end{equation}
These bounds will be determined using standard Sobolev embedding theorems
and the constants from the previous section, for each of the three
parameters~$\lambda$, $\sigma$, and~$\mu$. Notice that throughout this
section, we always assume $\lambda > 0$ and $\sigma \ge 0$, while the
mass~$\mu$ could be a real number of either sign.
\subsection{Variation of the short-range repulsion}
\label{subsec:lips}
We now state the Lipschitz estimates for the constructive implicit
function theorem in the case where~$\lambda$, the short-range repulsion
term, varies and the remaining parameters~$\mu$ and~$\sigma$ are held fixed.
\begin{lemma}[Lipschitz constants for variation of~$\lambda$]
Let $\lambda^* \in \R$ and $u^* \in \overline{\cH}^2$ be arbitrary, and
consider fixed positive constants~$d_{\lambda}$ and~$d_u$. Finally let~$\lambda$
and~$u$ be such that
\begin{displaymath}
  |\lambda-\lambda^*| \le d_{\lambda}
  \quad\mbox{ and }\quad
  \|u-u^*\|_{\overline{\cH}^2 } \le d_u \;. 
\end{displaymath}
Then the Lipschitz constants in~(\ref{eqn:lipschitz}) can be chosen as
\begin{displaymath}
  \begin{array}{rclcrcl}
    \DS M_1 &=& \DS \frac{\overline{C}_m f_{\max}^{(2)} (\lambda^* +
      d_\lambda)}{\pi^2} \; , & \qquad &
    \DS M_2 &=&  \DS \frac{\| f'(u^*+ \mu) \|_\infty}{\pi^2} +
      \frac{\sigma}{\pi^4} \; , \\[2ex]
    \DS M_3 &=& \DS \frac{f^{(1)}_{\max}}{\pi^2} +
      \frac{\sigma}{\pi^4} \; , & \qquad &
    \DS M_4 &=&  0 \; ,
  \end{array}
\end{displaymath}
where~$f_{\max}^{(1)}$ and~$f_{\max}^{(2)}$ are defined as
\begin{equation} \label{eqn:fpmax}
  f^{(p)}_{\max} =
  \max_{|\rho| \le \|u^*\|_\infty + \overline{C}_m d_u}
    |f^{(p)} (\rho + \mu)| \;.  
\end{equation}
These are well-defined since $f$ is a $C^2$-function. 
\end{lemma}
\begin{proof}
For our choice of constants~$d_\lambda$, $d_u$, reference parameter~$\lambda^* \in \R$
and function~$u^* \in \overline{\cH}^2$, and for arbitrary $v \in \overline{\cH}^2$,
assume that $|\lambda - \lambda^*| \le d_\lambda$ and $\| u - u^*\|_{\overline{\cH}^2}
\le d_u$. We start by deriving expressions for both~$M_1$ and~$M_2$. Notice that we have
\begin{eqnarray*}
  & & \hspace*{-2cm} \| D_uF(\lambda,u)[v] -
    D_uF(\lambda^*,u^*)[v] \|_{\overline{\cH}^{-2}} \\[1ex]
  &\le& \| \Delta ( \lambda f'(u+\mu) v - \lambda^* f'(u^*+\mu)v )
    \|_{\overline{\cH}^{-2}} + \sigma \, |\lambda - \lambda^*|
    \|v\|_{\overline{\cH}^{-2}} \\[1ex]
  &\le& \| \overline{P} ( \lambda f'(u+\mu) v - \lambda^* f'(u^*+\mu)v )
    \|_{\overline{\cH}^0} + \sigma \, |\lambda - \lambda^*| \, \frac{1}{\pi^4} \,
    \|v\|_{\overline{\cH}^2} \\[1ex]
  &\le& \|  \lambda f'(u+\mu) v - \lambda^* f'(u^*+\mu)v \|_{L^2} +
    \frac{\sigma}{\pi^4}\,|\lambda - \lambda^*| \, \|v\|_{\overline{\cH}^2} \\[1ex]
  &\le& \| \lambda f'(u+\mu) - \lambda^* f'(u^*+\mu)   \|_\infty \, \|v\|_{L^2} +
    \frac{\sigma}{\pi^4}\,|\lambda - \lambda^*| \, \|v\|_{\overline{\cH}^2} \\[1ex]
&\le& \left( \frac{1}{\pi^2} \|  \lambda f'(u+\mu)  - \lambda^* f'(u^*+\mu) \|_\infty
    + |\lambda - \lambda^*| \, \frac{\sigma}{\pi^4}\right) \|v\|_{\overline{\cH}^2} \; . 
\end{eqnarray*}
The first estimate follows straightforwardly from the definition of the Fr\'echet
derivative~(\ref{eqn:deffrechetderivative}), while the second one uses the fact that
the Laplacian is an isometry (cf.\ Lemma~\ref{lemma:lapiso}) and the Banach scale
estimate between~$\overline{\cH}^{-2}$ and~$\overline{\cH}^{2}$ (cf.\ Lemma~\ref{lemma:bscaleest}).
The third estimate follows from~$\|\overline{P}\| = 1$, as well as the fact
that~$\overline{\cH}^0$ and~$L^2(\Omega)$ are equipped with the same norm. Finally,
the fourth estimate is straightforward, and the factor~$1/\pi^2$ in the fifth estimate
follows from $v \in \overline{\cH}^{2} \subset \overline{\cH}^0$ and the estimate in
Lemma~\ref{lemma:bscaleest}.

The above estimate shows that the operator norm of the difference of the two
Fr\'echet derivatives is bounded by the expression in parentheses. The first of
these two terms will now be estimated further. For this, note first that
\begin{eqnarray*}
  & & \| \lambda f'(u+\mu) - \lambda^* f'(u^*+\mu) \|_\infty \\[1ex]
  & & \qquad\qquad \le \;
    |\lambda| \, \| f'(u+\mu) - f'(u^*+\mu) \|_\infty +
    | \lambda-\lambda^*| \, \| f'(u^*+\mu) \|_\infty \; . 
\end{eqnarray*}
For fixed $x \in \Omega$, we know from the mean value theorem that there exists
a number~$\xi(x)$ between~$u(x)$ and~$u^*(x)$ such that
\begin{displaymath}
  | f'(u(x)+\mu) -  f'(u^*(x)+\mu) | \le
    |f''(\xi(x)+\mu)| \; |u(x)-u^*(x)| \;.
\end{displaymath}
Since~$\xi(x)$ is contained between~$u(x)$ and~$u^*(x)$ for all
$x \in \Omega$, the function~$\xi$ is bounded. Combining this fact with
the definition of~$\overline{C}_m$ in~(\ref{eqn:cmcb}) we get
\begin{displaymath}
  \| \xi \|_\infty \le \| u^*\|_\infty + \| u - u^* \|_\infty \le
  \|u^*\|_\infty + \overline{C}_m \|u - u^*\|_{\overline{\cH}^2} \le
  \|u^*\|_\infty + \overline{C}_m d_u \;,
\end{displaymath}
and therefore
\begin{eqnarray*}
  & & \hspace*{-2cm}
  \| \lambda f'(u+\mu) - \lambda^* f'(u^*+\mu) \|_\infty \\[1ex]
  & \le &
    |\lambda| \, f^{(2)}_{\max}\, \| u - u^*\|_\infty +
    |\lambda-\lambda^*|\, \| f'(u^*+\mu) \|_\infty \\[1ex]
  & \le & |\lambda| \, f^{(2)}_{\max}\, \overline{C}_m \, 
    \| u - u^*\|_{\overline{\cH}^2} +
    |\lambda-\lambda^*|\, \| f'(u^*+\mu) \|_\infty \;, 
\end{eqnarray*}
where~$f^{(2)}_{\max}$ is defined in~(\ref{eqn:fpmax}). Incorporating this
into the previous estimate, we see that
\begin{eqnarray*}
  & & \hspace*{-1.2cm}
    \| D_uF(\lambda,u) - D_uF(\lambda^*,u^*) \|_{\cL(\overline{\cH}^2,
    \overline{\cH}^{-2})} \\[1ex]
  &\le& \left( \frac{\overline{C}_m \, f^{(2)}_{\max} \,
    (\lambda^* + d_{\lambda})}{\pi^2} \right) \|u - u^*\|_{\overline{\cH}^2} +
    \left( \frac{\|f'(u^* + \mu)\|_\infty}{\pi^2} +
    \frac{\sigma}{\pi^4}\right) |\lambda - \lambda^*| \; . 
\end{eqnarray*}
This equation directly gives the values of the Lipschitz constants~$M_1$ and~$M_2$
given in the statement of the lemma.

We now turn our attention to the remaining constants~$M_3$ and~$M_4$. The Fr\'echet
derivative of $F$~with respect to~$\lambda$ is given by
\begin{displaymath}
  D_\lambda F(\lambda,u) = - \Delta f(u+ \mu) - \sigma u \;. 
\end{displaymath}
Using almost identical steps as the calculation of~$M_1$ and~$M_2$, we get
\begin{eqnarray*}
  & & \hspace*{-2cm}
    \|D_\lambda F(\lambda,u) - D_\lambda
    F(\lambda^*,u^*)\|_{\overline{\cH}^{-2}} \\[1ex]
  & \le & \| \Delta (f(u+\mu)-f(u^*+\mu))\|_{\overline{\cH}^{-2}} +
    |\sigma| \, \| u - u^*\|_{\overline{\cH}^{-2}} \\[1ex]
  &\le& \| f(u+\mu)-f(u^*+\mu)\|_{L^2} + \frac{\sigma}{\pi^4} \,
    \| u - u^*\|_{\overline{\cH}^{2}} \\[1ex]
  &\le& f^{(1)}_{\max} \, \| u - u^*\|_{L^2} + \frac{\sigma}{\pi^4} \,
    \| u - u^*\|_{\overline{\cH}^{2}} \\[1ex]
  &\le& \left( \frac{f^{(1)}_{\max}}{\pi^2} + \frac{\sigma}{\pi^4} \right)
    \| u - u^*\|_{\overline{\cH}^{2}} \;.
\end{eqnarray*}
Notice that in estimating the norm of this difference of Fr\'echet derivatives
we use the standard identification of~$\cL(\R,\overline{\cH}^{-2})$
with~$\overline{\cH}^{-2}$. Furthermore, in the above inequalities, we
have made liberal use of the constructive Sobolev embedding results from
the previous section. This gives the constants~$M_3$ and~$M_4$ given in
the statement of the lemma. 
\end{proof}
\subsection{Variation of the long-range elasticity}
\label{subsec:sigma}
We now establish Lipschitz constants for the case when the parameter~$\sigma$
varies and both~$\lambda$ and~$\mu$ are held fixed. 
\begin{lemma}[Lipschitz constants for variation of~$\sigma$]
Let $\sigma^* \in \R$ and $u^* \in \overline{\cH}^2$ be arbitrary, and
consider fixed positive constants~$d_{\sigma}$ and~$d_u$. Finally let~$\sigma$
and~$u$ be such that
\begin{displaymath}
  |\sigma-\sigma^*| \le d_{\sigma}
  \quad\mbox{ and }\quad
  \|u-u^*\|_{\overline{\cH}^2 } \le d_u \;. 
\end{displaymath}
Then the Lipschitz constants in~(\ref{eqn:lipschitz}) can be chosen as
\begin{displaymath}
  M_1 \; = \; \frac{\lambda \, f^{(2)}_{\max} \, \overline{C}_m}{\pi^2} \; ,
  \qquad
  M_2 \; = \; M_3 \; = \; \frac{\lambda}{\pi^4} \; ,
  \qquad
  M_4 \; = \; 0\; ,
\end{displaymath}
where the value of~$f^{(2)}_{\max}$ is defined in~(\ref{eqn:fpmax}).
\end{lemma}
\begin{proof}
We start by computing the constants~$M_1$ and~$M_2$. Holding~$\mu$ and~$\lambda > 0$
fixed in the equation for~$D_uF$, we are able to follow very similar arguments as in
the $\lambda$-varying case, including the use of the Sobolev embedding formulas and
the mean value theorem. The resulting estimate is given by
\begin{eqnarray*}
  & & \hspace*{-2cm} \| D_uF(\sigma,u)[v] -  D_uF(\sigma^*,u^*)[v]
    \|_{\overline{\cH}^{-2}} \\[1ex]
  & \le & \| \Delta (\lambda (f'(u + \mu) - f'(u^*+\mu)) v) \|_{\overline{\cH}^{-2}} +
    \lambda \, |\sigma-\sigma^*| \, \|v\|_{\overline{\cH}^{-2}} \\[1ex]
  & \le & \lambda \, \|f'(u+\mu) - f'(u^*+\mu)\|_\infty \, \| v \|_{L^2} +
    \lambda \, |\sigma-\sigma^*| \, \|v\|_{\overline{\cH}^{-2}}\\[1ex]
  & \le &\left( \frac{ \lambda \, f^{(2)}_{\max} \,  \overline{C}_m}{\pi^2} \right)
    \|u-u^* \|_{\overline{\cH}^2} \, \| v \|_{\overline{\cH}^2} +
    \left(\frac{\lambda}{\pi^4} \right) \, |\sigma-\sigma^*| \,
    \|v\|_{\overline{\cH}^{2}}\;. 
\end{eqnarray*}
This establishes constants~$M_1$ and~$M_2$ given in the lemma. 
We now turn our attention to the constants~$M_3$ and~$M_4$. 
The derivative of~$F$ with respect to~$\sigma$ is given by
\begin{displaymath}
  D_\sigma F(\sigma,u) = - \lambda u \;. 
\end{displaymath}
Therefore, once again Lemma~\ref{lemma:bscaleest},
we get
\begin{displaymath}
  \|D_\sigma F(\sigma,u) - D_\sigma F(\sigma^*,u^*)\|_{\overline{\cH}^{-2}}
  \; \le \;
  \lambda \, \|u-u^*\|_{\overline{\cH}^{-2}}
  \; \le \;
  \frac{\lambda}{\pi^4} \, \|u-u^*\|_{\overline{\cH}^2 }\; ,
\end{displaymath}
which gives the constants~$M_3$ and~$M_4$ stated in the lemma. 
\end{proof}
\subsection{Varying the relative proportion of the two polymers}
\label{subsec:mu}
In this final subsection we now consider the third parameter variation,
namely that of~$\mu$. 
\begin{lemma}[Lipschitz constants for variation of~$\mu$]
Let $\mu^* \in \R$ and $u^* \in \overline{\cH}^2$ be arbitrary, and
consider fixed positive constants~$d_{\mu}$ and~$d_u$. Finally let~$\mu$
and~$u$ be such that
\begin{displaymath}
  |\mu-\mu^*| \le d_{\mu}
  \quad\mbox{ and }\quad
  \|u-u^*\|_{\overline{\cH}^2 } \le d_u \;. 
\end{displaymath}
Then the Lipschitz constants in~(\ref{eqn:lipschitz}) can be chosen as
\begin{displaymath}
  M_1 \; = \; \frac{\lambda \, f^{(2)}_{\max,\mu} \, \overline{C}_m}{\pi^2} \; ,
  \qquad
  M_2 \; = \; M_3 \; = \; \frac{\lambda \, f^{(2)}_{\max,\mu}}{\pi^2} \; ,
  \qquad
  M_4 \; = \; \lambda \; f^{(2)}_{\max,\mu} \; ,
\end{displaymath}
where the constant~$f^{(2)}_{\max, \mu}$ is defined as
\begin{equation}\label{eqn:fpmumax}
  f^{(2)}_{\max, \mu} = \max_{ |\rho| \le \| u^* + \mu^*\|_{\infty} +
  \overline{C}_m d_u + d_\mu   } |f''(\rho)| \; . 
\end{equation}
\end{lemma}
\begin{proof}
Using a similar format to the last two proofs, we consider~$\lambda > 0$
and $\sigma \ge 0$ to be fixed constants and only allow~$\mu$ to vary. 
The we have
\begin{eqnarray*}
  & & \hspace*{-2cm} \| D_uF(\mu,u)[v] - D_uF(\mu^*,u^*)[v]
    \|_{\overline{\cH}^{-2}} \\[1ex]
  &\le& \| \Delta ( \lambda ( f'(u + \mu) - f'(u^* + \mu^*))v)
    \|_{\overline{\cH}^{-2}} \\[1ex]
  &\le& \lambda \| f'(u + \mu) - f'(u^* + \mu^*)\|_{\infty} \,
    \| v \|_{L^2} \\[1ex]
  &\le& \frac{\lambda}{\pi^2} \| f'(u + \mu) - f'(u^* + \mu^*)\|_{\infty} \,
    \| v \|_{\overline{\cH}^{2}} \; . 
\end{eqnarray*}
As in the previous calculations, we use the mean value theorem to bound
the value of the maximum norm $\| f'(u + \mu) - f'(u^* + \mu^*)\|_{\infty}$. 
To do so, note that if a real value~$\rho$ is between the two
numbers~$u^*(x)+\mu$ and~$u(x) + \mu^*$ for some $x \in \Omega$,
then one has
\begin{eqnarray*}
  |\rho| &\le& \| u + \mu^*\|_{\infty} + |\mu - \mu^*| \\[1ex]
  &\le& \|u^*+\mu^*\|_\infty + \| u - u^*\|_\infty + |\mu - \mu^*| \\[1ex]
  &\le& \|u^*+\mu^*\|_\infty + \overline{C}_m \| u - u^*\|_{\overline{\cH}^2}
    + |\mu - \mu^*| \le \|u^*+\mu^*\|_\infty + \overline{C}_m d_u + d_\mu \; .
\end{eqnarray*}
Thus, by the mean value theorem, followed by the use of our Sobolev embedding
results, one further obtains
\begin{eqnarray*}
  \| f'(u + \mu) - f'(u^* + \mu^*)\|_\infty &\le&
    f^{(2)}_{\max, \mu} \, \| (u + \mu) - (u^* + \mu^*) \|_\infty \\[1ex]
  &\le& f^{(2)}_{\max, \mu} \, \left( \overline{C}_m
    \|u-u^*\|_{\overline{\cH}^2} + |\mu-\mu^*| \right) \; , 
\end{eqnarray*}
and combining this with our previous estimate we finally deduce
\begin{displaymath}
  \| D_uF(\mu,u) -  D_uF(\mu^*,u^*) \|_{\cL(\overline{\cH}^{-2},
    \overline{\cH}^{2})}
  \; \le \;
  \frac{\lambda \; f^{(2)}_{\max,\mu} }{\pi^2}
    \left( \overline{C}_m \|u-u^*\|_{\overline{\cH}^2} +
    |\mu-\mu^*| \right) \; . 
\end{displaymath}
This gives the constants~$M_1$ and~$M_2$. We now look at the bounds
for~$M_3$ and~$M_4$. The derivative of~$F$ with respect to $\mu$ is
given by
\begin{displaymath}
  D_\mu F(\mu,u) = -\Delta (\lambda f'(u + \mu)) \;. 
\end{displaymath}
By similar reasoning as before, we then get
\begin{eqnarray*}
  \|D_\mu F(\mu,u) - D_\mu F(\mu^*, u^*)\|_{\overline{\cH}^{-2}}
    &=& \lambda\, \| \Delta(f'(u + \mu) -
    f'(u^*+\mu^*))\|_{\overline{\cH}^{-2}} \\[1ex]
  &\le& \lambda\, \| f'(u + \mu) - f'(u^* + \mu^*)\|_{L^2} \\[1ex]
  &\le& \lambda\, f_{\max,\mu}^{(2)} \; \|(u + \mu) -
    (u^* + \mu^*)\|_{L^2} \\[1ex]
  &\le& \lambda \, f_{\max,\mu}^{(2)} \, \left( \frac{1}{\pi^2}
    \|u-u^*\|_{\overline{\cH}^2} + |\mu-\mu^*| \right) \; . 
\end{eqnarray*}
This gives the constants~$M_3$ and~$M_4$ and completes the proof of
the lemma. 
\end{proof}
With the above lemma we have completed the discussion of all of the
Lipschitz constant bounds for all three equation parameters. 
\section{Illustrative examples}
\label{sec:eg}
In this section, we present some examples of validated equilibrium solutions 
in order to illustrate the power of our theoretical validation method. In
particular, the theoretical methods developed above can be used to produce
a {\em validated region\/} in parameter cross phase space. We emphasize that
this section is only intended to present proof of concept. We have not made
any attempt to optimize our results or to add computational methods to speed
up the code. For example, the interval arithmetic package INTLAB~\cite{rump:99a}
that we have used is not written in parallel, and we have not attempted to parallelize
any of our algorithms. As another example, in the past we have found that careful
preconditioning can speed up the computation time significantly. Rather than
add any of these techniques at this stage, we have chosen to reserve numerical
considerations for a future paper, in which we will also address additional
questions such as how to use these methods iteratively to validate branches
of solutions.
\begin{figure} \centering
  \includegraphics[width=0.7\textwidth]{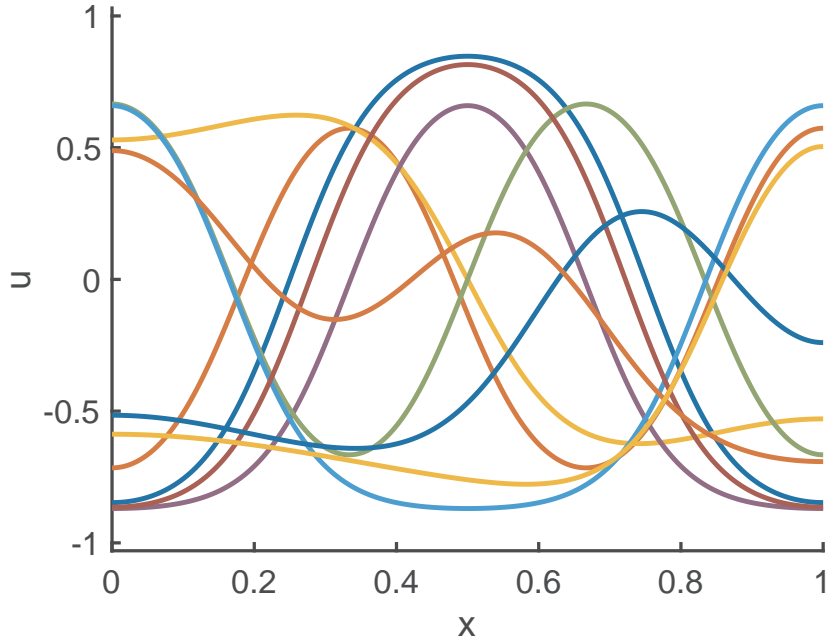}
  \caption{\label{fig:1d}
           Ten sample validated one-dimensional equilibrium solutions.
           For all solutions we choose $\lambda = 150$ and $\sigma = 6$.
           Three of the solutions have total mass $\mu=0$, three are for
           mass $\mu = 0.1$, three for $\mu = 0.3$, and finally one
           for $\mu = 0.5$.}
\end{figure}
\begin{figure} \centering
  \includegraphics[width=0.45\textwidth]{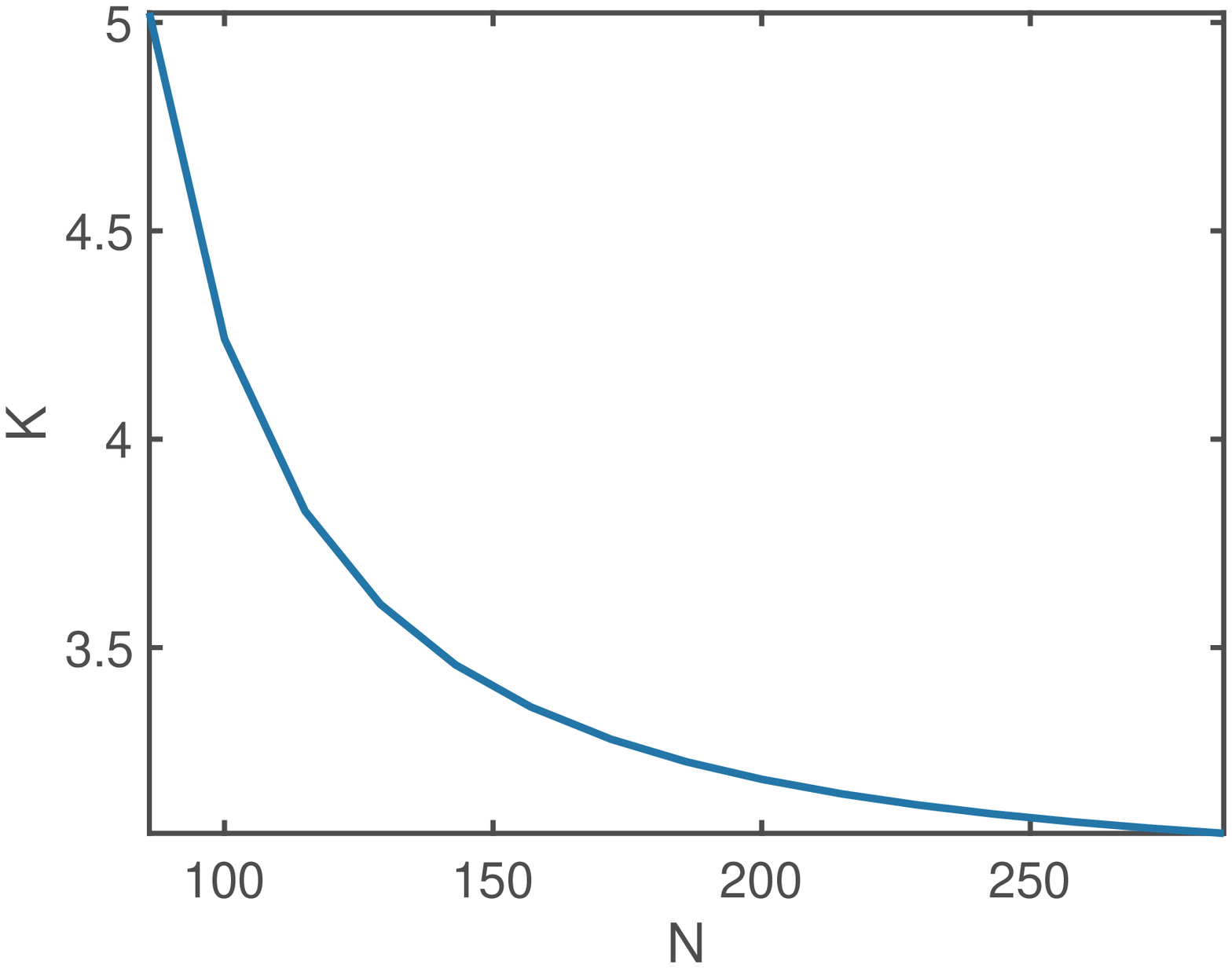}
  \hspace*{0.5cm}
  \includegraphics[width=0.45\textwidth]{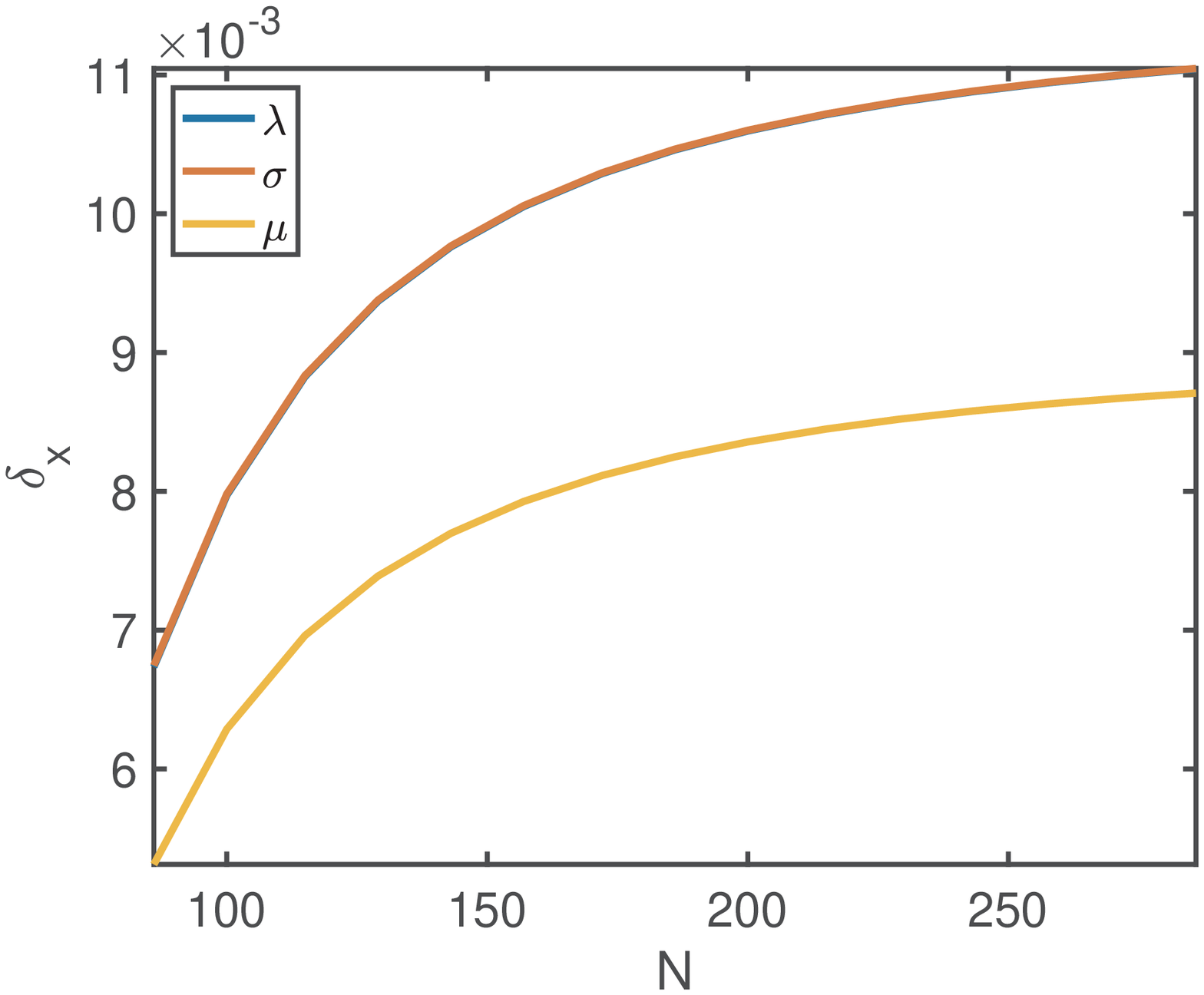} \\[2ex]
  \includegraphics[width=0.45\textwidth]{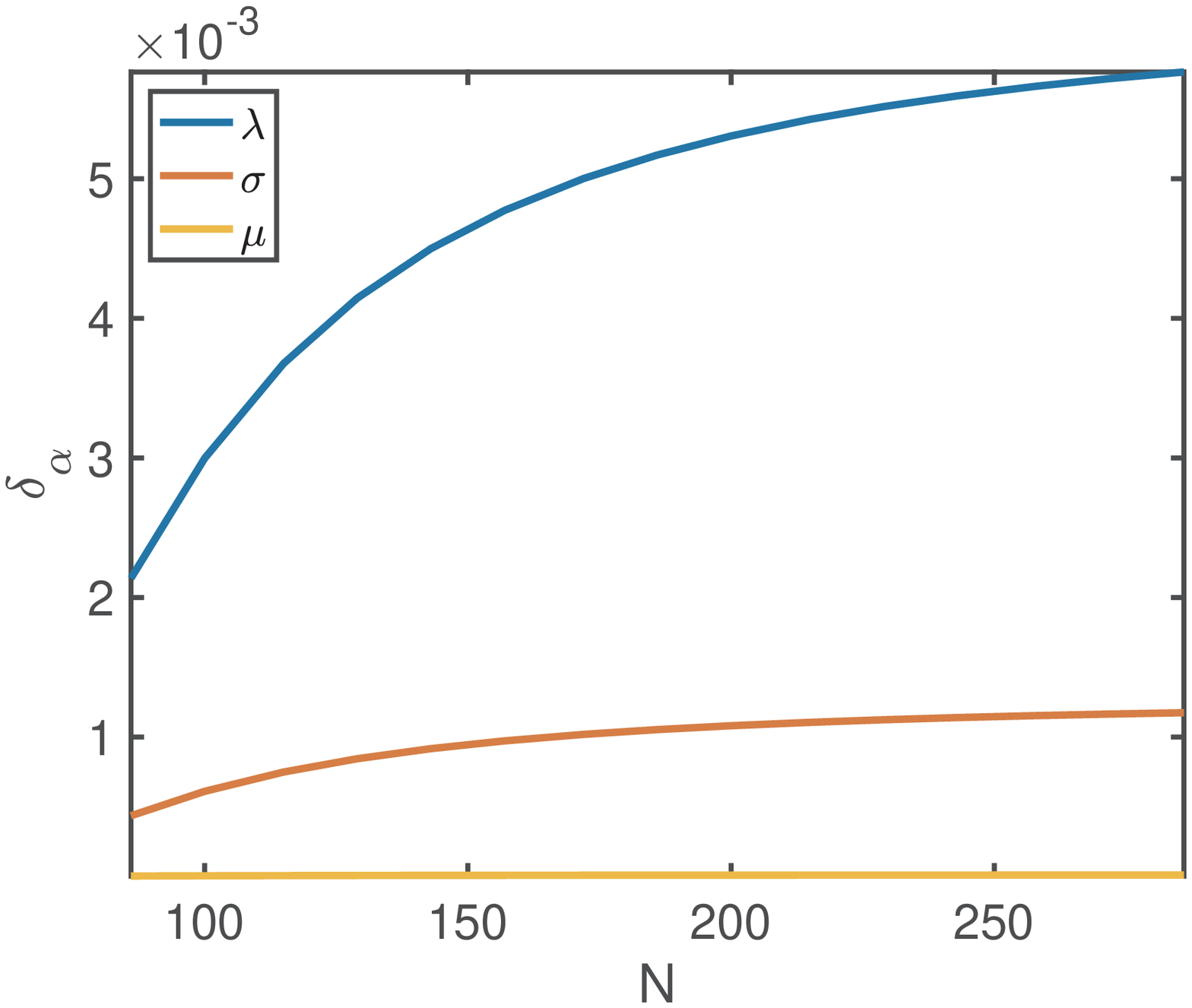}
  \caption{\label{fig:tradeoff}
           There is a tradeoff between high-dimensional calculations and
           optimal results. The top left figure shows how the bound of~$K$
           varies with the dimension of the truncated approximation matrix
           used to calculate~$K_N$. These calculations are for dimension one,
           but a similar effect occurs in higher dimensions as well. The top
           right figure shows the corresponding estimate for~$\delta_x$, and
           the bottom panel shows the estimate for~$\delta_\alpha$,
           where~$\alpha$ is each of the three parameters. The size of
           the validated interval grows larger as the truncation dimension
           grows, but with diminishing returns on the computational
           investment.}
\end{figure}
\begin{table}
  \begin{tabular}{|c|c|c||c|c|c|}
  \hline
  $\mu$ & $K$ & $N$ & $P$  & $\delta_\alpha$  & $\delta_x$ \\
  \hline\hline
    $0$ & 6.2575  & 89	& $\lambda$ & 0.0016     & 0.0056   \\
                  &			&& $\sigma$  & 2.9259e-04 & 0.0056 \\
                  &			&& $\mu$     & 2.8705e-06 & 0.0044 \\ \hline
    $0.1$ & 6.4590 & 104 & $\lambda$ & 0.0011     & 0.0050  \\
            &			&& $\sigma$  & 2.5369e-04 & 0.0050 \\
            &			&& $\mu$     & 2.5579e-06 &  0.0041  \\ \hline      
    $0.5$ & 3.1030  & 74 & $\lambda$ & 0.0052    & 0.0107   \\
            &			&& $\sigma$  &  0.0011  &  0.0106  \\
            &			&& $\mu$     & 1.2871e-05  & 0.0092 \\ \hline
  \end{tabular}
  \vspace*{0.3cm}
  \caption{\label{table:1d}
           A sample of the one-dimensional solution validation
           parameters for three typical solutions. In each case, we
           use $\sigma  = 6$ and~$\lambda = 150$. If we had chosen a
           larger value of~$N$, we could significantly improve the results.}
\end{table}

Under the hypotheses of Theorem~\ref{nift:thm}, the constructive implicit
function theorem, for each~$\delta_\alpha$ and~$\delta_x$ satisfying both
parts of~(\ref{nift:thm2}), we are guaranteed that the solution is uniquely
contained in the corresponding $(\delta_\alpha,\delta_x)$-box, where~$\alpha$
is the chosen of the three parameters. In fact, if we fix~$\delta_\alpha$
small enough, then there are a range of values of~$\delta_x$ bounded below
by the quadratic second equation and above by the linear first equation.
We can view the region bounded by the lower limit of~$\delta_x$ as an
{\em accuracy region\/}, within which the equilibrium is guaranteed to lie;
and the region bounded by the upper limit of~$\delta_x$ is a {\em uniqueness
region\/}, which contains the accuracy region, within which the solution is
guaranteed to be unique. If~$\delta_\alpha$ is chosen to be the point for
which the line and curve in~(\ref{nift:thm2}) intersect, then this is the
largest possible value of~$\delta_\alpha$ for which the theorem holds, and
the accuracy and uniqueness regions coincide. In our calculations we have
validated using this maximal interval in parameter space, and we have done
the calculation of the interval size for each of the three parameters.
\begin{figure} \centering
  \includegraphics[width=0.3\textwidth]{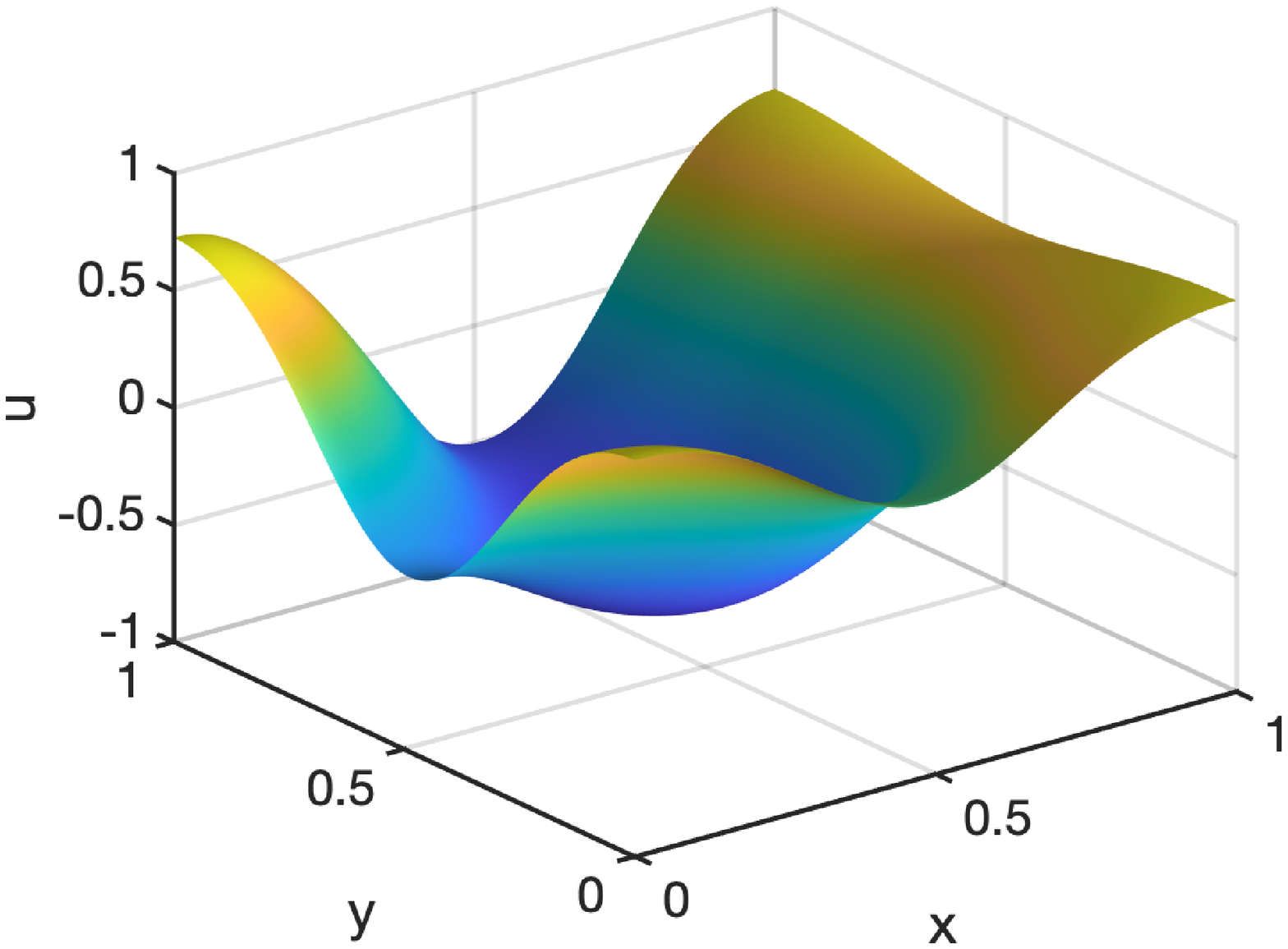}
  \includegraphics[width=0.3\textwidth]{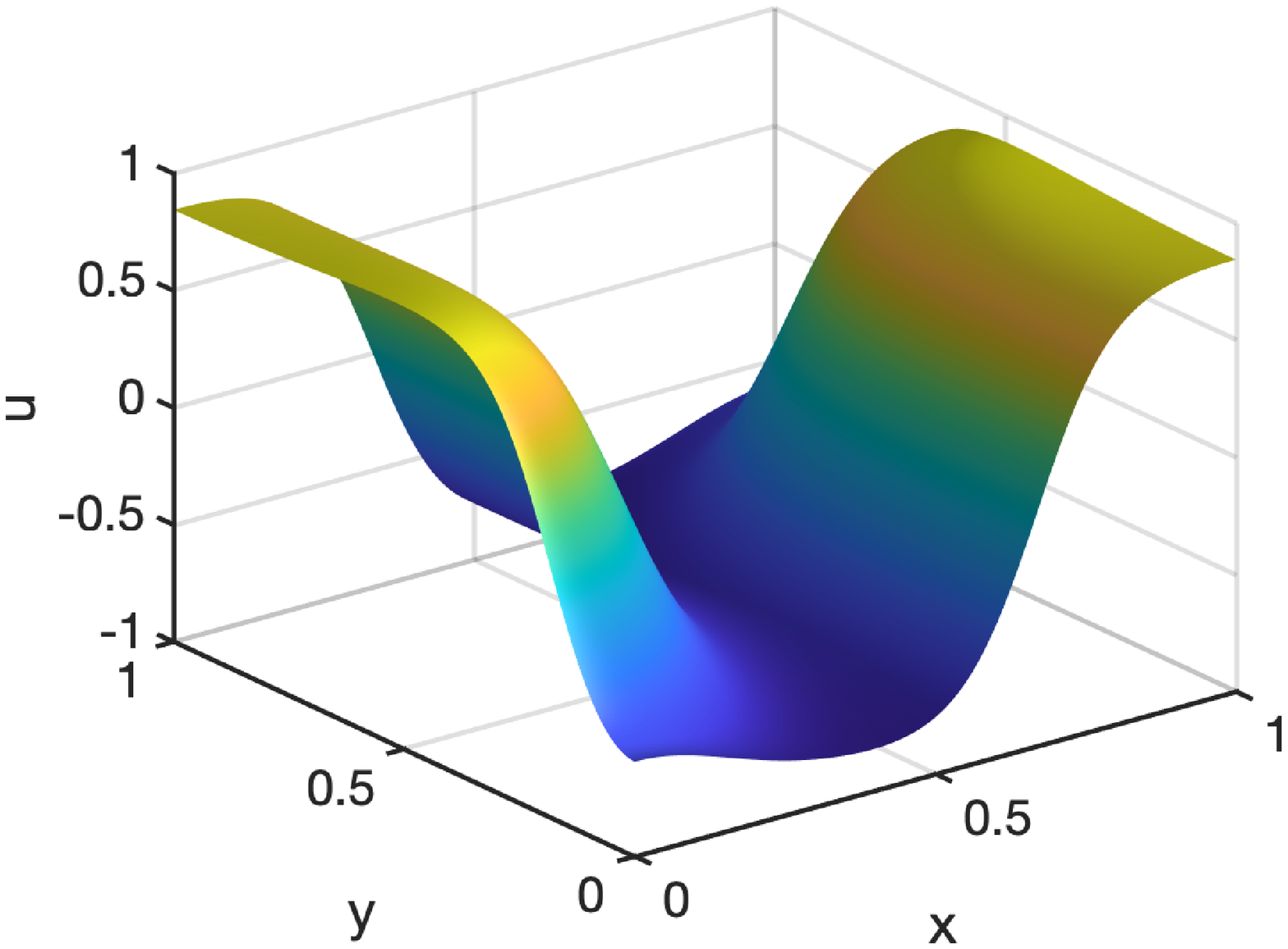}
  \includegraphics[width=0.3\textwidth]{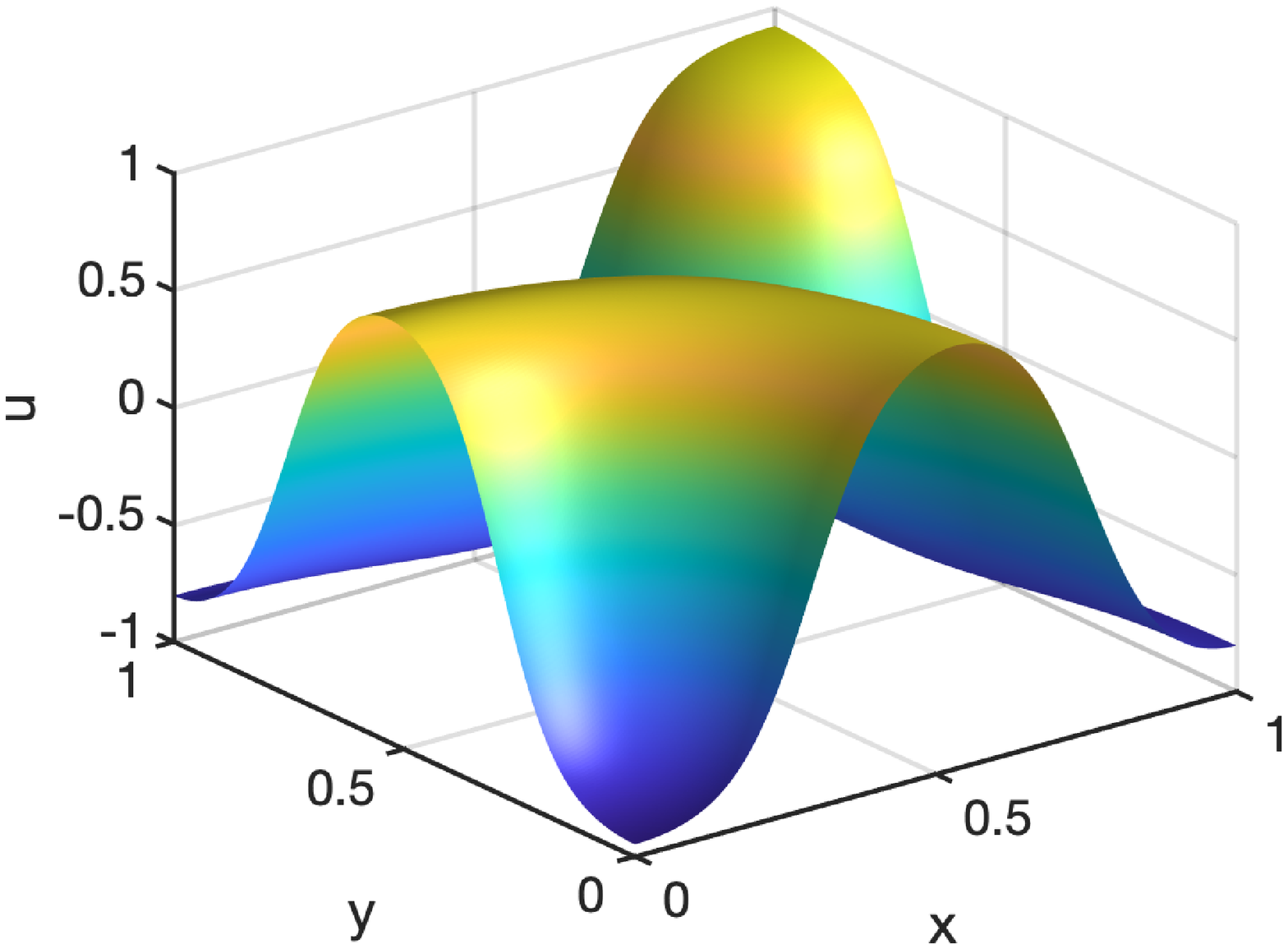}\\[2ex]
  \includegraphics[width=0.3\textwidth]{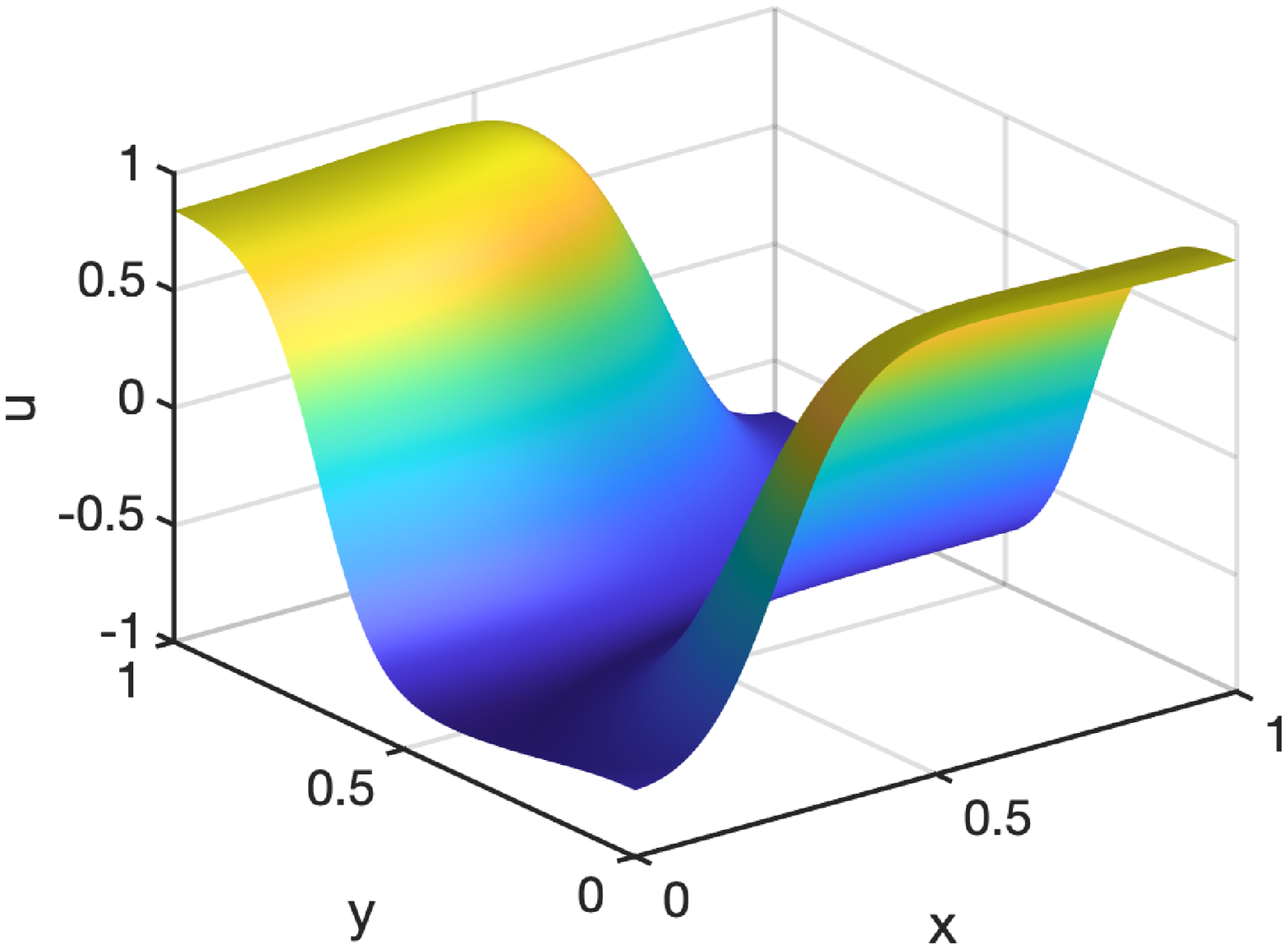}
  \includegraphics[width=0.3\textwidth]{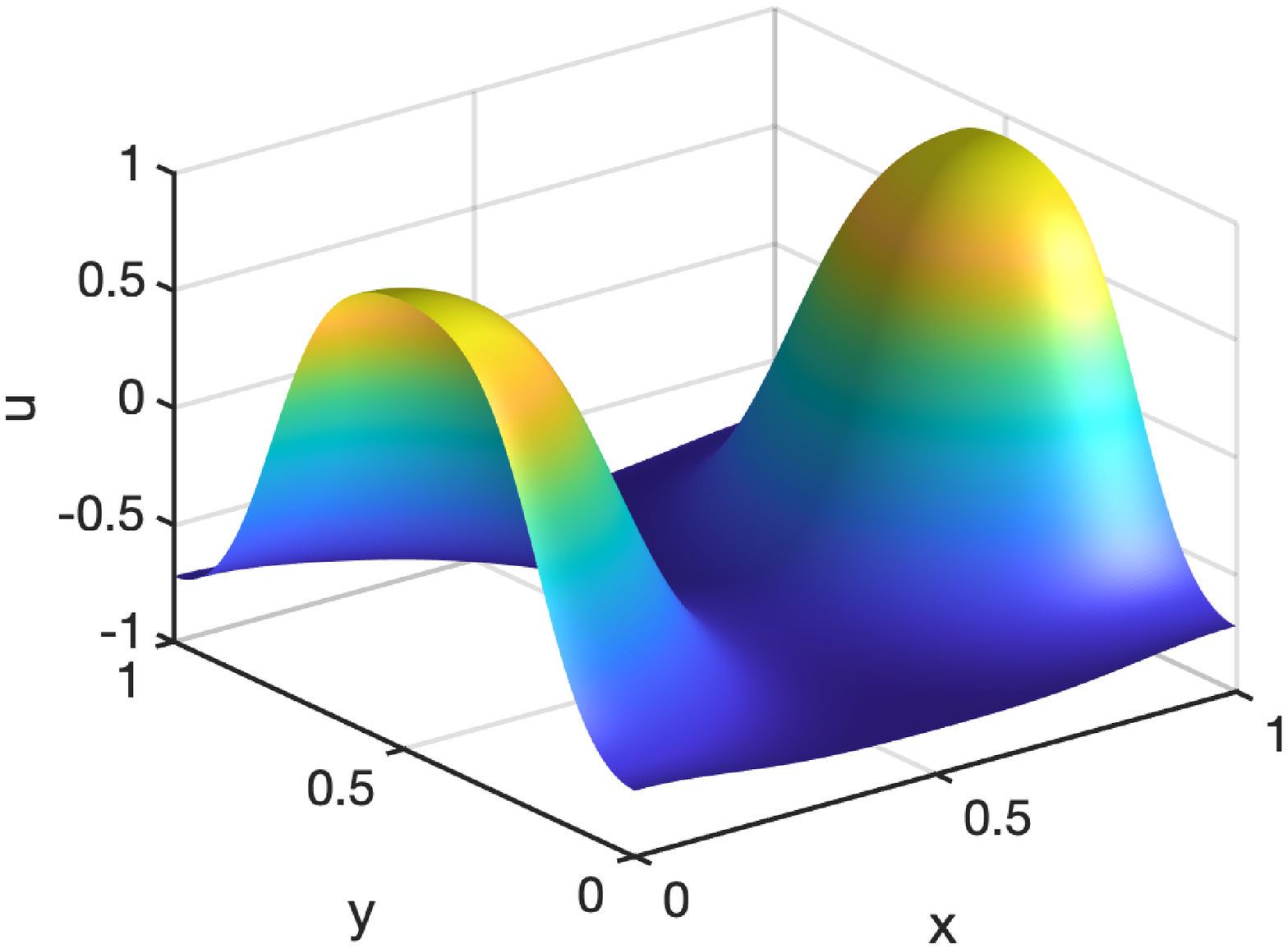}
  \includegraphics[width=0.3\textwidth]{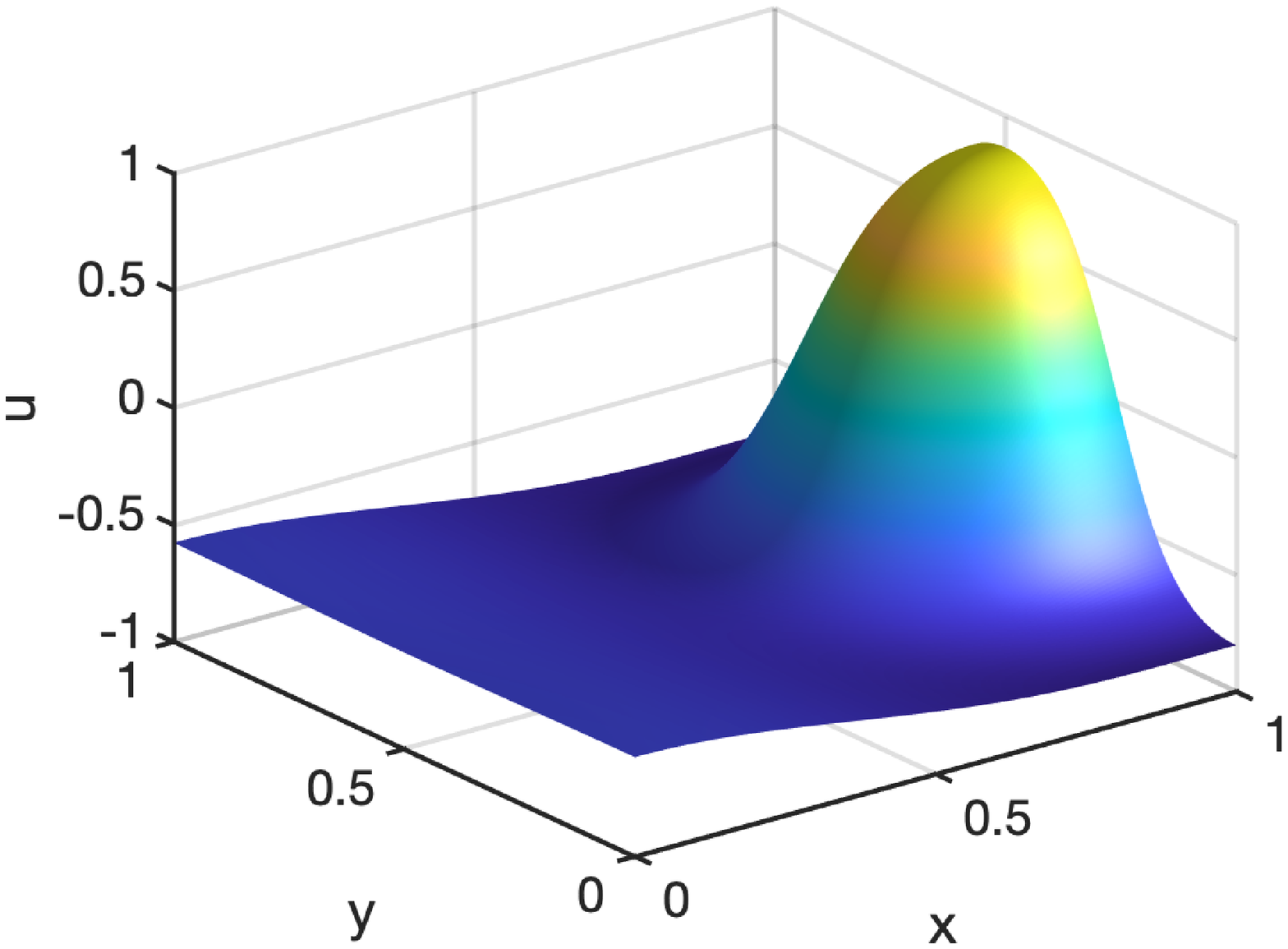}
  \caption{\label{fig:2d}
           Six of the seventeen validated two-dimensional equilibrium solutions.
           For all seventeen solutions we use $\sigma = 6$. Five of these solutions
           are for $\lambda = 75$ and $\mu = 0$ (top left). The rest of them use
           $\lambda = 150$ and $\mu = 0$ (top middle and top right), $\mu =0.1$
           (bottom left), $\mu =0.3$ (bottom middle), and $\mu = 0.5$ (bottom
           right).}
\end{figure}
\begin{table}
  \begin{tabular}{|c|c|c||c|c|c|}
    \hline
    $(\lambda,\mu)$ & $K$ & $N$ & $P$ & $\delta_\alpha$ & $\delta_x$ \\
    \hline\hline            
    $(75,0)$ & 21.1303  & 28 & $\lambda$ & 1.6124e-04   & 0.0020 \\
            &			&& $\sigma$  &6.1338e-05  &0.0020  \\
            &			&& $\mu$     &  5.9914e-07 & 0.0016 \\ \hline
    $(150,0.1)$ & 30.1656  & 72 & $\lambda$ & 1.1833e-05   & 4.7710e-04\\
            &			&& $\sigma$  &5.1514e-06   & 4.7858e-04  \\
            &			&& $\mu$     & 4.4558e-08  &4.2316e-04 \\ \hline
  \end{tabular}
  \vspace*{0.3cm}
  \caption{\label{table:2d}
           A sample of the two-dimensional validation parameters
           for a couple of typical solutions. In all cases, we use
           $\sigma = 6$. Again as in the previous table, we could
           improve results by choosing a larger value of~$N$, but
           in this case since~$N$ is only the linear dimension, the
           dimension of the calculation varies with~$N^2$.}
\end{table}

We have validated ten different equilibrium solutions in one dimension,
shown in Figure~\ref{fig:1d}. Some examples of the associated validation
parameters are presented in Table~\ref{table:1d}. Ideally, we are able to
validate the largest possible $(\delta_\alpha,\delta_x)$-box in which we
can guarantee that the solution exists. However, there is a tradeoff between 
computational cost and optimal bounds. The most computationally costly part
of our estimates is the calculation of~$K_N$, the bound on the inverse of the
linearization of the truncated system. As depicted in Figure~\ref{fig:tradeoff},
the bounds on~$K$, and correspondingly on~$\delta_x$ and~$\delta_\alpha$, depend
significantly on the value of~$N$ that is chosen for the truncation dimension.
Since our goal is to use these validations iteratively for path following, we
will not be able to refine our calculations each time. Therefore as a rule of
thumb for a starting point, we used the equation in Theorem~\ref{thm:k} to
guess that we would have a successful validation for $N \approx C
\|q\|^{1/2}_{H^2}$, where~$C$ is a fixed order one constant. In our
calculations for the ten solutions, this results in a dimension that
varies. For these calculations we chose~$N$ values ranging between~50
and~200. The values of~$M_i$ become progressively larger as you go
from~$\lambda$ to~$\sigma$ to~$\mu$. This means that the corresponding
values of~$\delta_\alpha$ are worse (i.e., smaller), respectively, often
by one or two orders of magnitude. However, the values of~$\delta_x$ for
the three cases are of the same order. While we could increase~$N$ to
improve the estimates, Figure~\ref{fig:tradeoff} shows that there are
diminishing returns on computational investment, and eventually at some~$N$,
we could not have done much better even with a significantly larger value
of~$N$.
\begin{figure} \centering
  \includegraphics[width=0.7\textwidth]{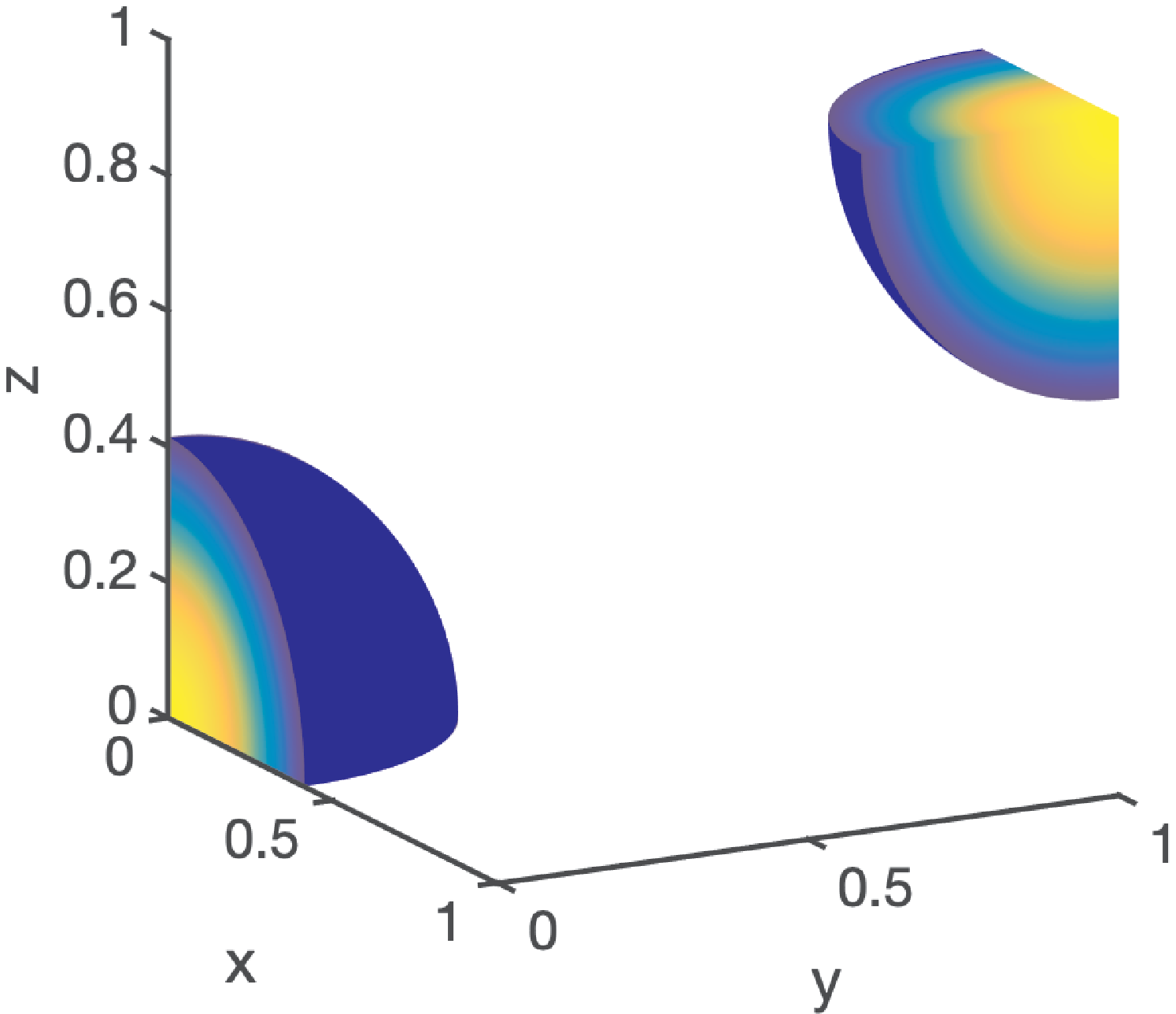}
  \caption{\label{fig:3d}
           A three-dimensional validated solution for the parameter
           values $\lambda = 75$, $\sigma = 6$, and~$\mu = 0$.}
\end{figure}
\begin{table}
  \begin{tabular}{|c|c|c||c|c|c|}
    \hline
    $(\lambda,\sigma,\mu)$ & $K$ & $N$ & $P$ & $\delta_\alpha$ & $\delta_x$ \\           
    \hline\hline
    $(75,6,0)$ & 22.6527 & 22 & $\lambda$ & 0.1143e-04  & 0.5917e-03 \\
            &			&& $\sigma$  & 0.1707e-04   &   0.5955e-03  \\
            &			&& $\mu$     &  0.0010e-04 &  0.4901e-03 \\ \hline
  \end{tabular}
  \vspace*{0.3cm}
  \caption{\label{table:3d}
           Validation parameters for a three-dimensional sample solution.}
\end{table}

In two dimensions, we have validated seventeen different solutions for
varying parameter values. A representative sample are given in
Figure~\ref{fig:2d}, with some sample validation parameters presented
in Table~\ref{table:2d}. Again here, there is a tradeoff between
computational speed and optimal results, but with all of the computations
being significantly longer due to the increased dimension; if the
function~$u$ is encoded by a Fourier coefficient array of size $N \times N$,
then the derivative matrix is of size~$(N^2 - 1)^2$, where the~$-1$ is due
to the fact that we have removed the constant term. As in one dimension,
the resulting~$\delta_\alpha$ values vary significantly, but the~$\delta_x$
values do not. Figure~\ref{fig:3d} and Table~\ref{table:3d} show the details
of a solution which is validated in three dimensions, with much the same
observed behavior. Three-dimensional result validation requires a much
larger computational effort, since if the function~$u$ is given by a
Fourier coefficient array of size $N \times N \times N$, then the
derivative matrices with inverse being approximated are of size~$(N^3-1)^2$.
\section{Conclusions}
\label{sec:theend}
As outlined in more detail in the introduction, in this paper we presented
the theoretical foundations for validating branch segments of equilibrium 
solutions for the diblock copolymer model. Our approach is based on using the 
natural Sobolev norms which are used in the study of the underlying evolution
equation, and they have been derived in all three relevant physical dimensions. As
a side result, we obtained a method based on Neumann series to determine 
rigorous upper bounds on the inverse Fr\'echet derivative of the diblock copolymer 
operator which are of interest in their own right, as they are connected to
the pseudo-spectrum of this non-self-adjoint operator, see~\cite{trefethen:embree:05a}.
Moreover, we have demonstrated briefly in the last section how these results can be
used to obtain computer-assisted proofs for selected diblock copolymer equilibrium
solutions.

While the present paper is a first step towards a complete path-following
framework for the diblock copolymer model in dimensions up to three, there 
are still a number of issues that have to be addressed. On the theoretical side,
one has to develop a pseudo-arclength continuation method with associated linking
conditions which operates in an automatic fashion. This can be done by using
the constructive implicit function theorem as a tool, similar to the applications
to slanted box continuation and limit point resolution which were presented
in~\cite[Sections~2.2 and~2.3]{sander:wanner:16a}. In addition, the bottleneck
in the current validation step is the estimation of the norm bound for the inverse.
Especially in two, and even more so in three dimensions, one has to implement
path-following in such a way that the estimate does not have to be validated at
every step. This can be accomplished via perturbation arguments, and further
speedups are possible by using the sparseness of the involved matrices. However,
all of these issues are nontrivial and lie beyond the scope of the current paper
--- they will therefore be presented elsewhere.
\section*{Acknowledgments}
%
%
We thank the referee for  helpful comments, which improved the quality of this paper.
E.S. and T.W.~were partially supported by NSF grant DMS-1407087.
E.S. was partially supported by NSF grant DMS-1440140 while in
residence at the Mathematical Sciences Research Institute in Berkeley, 
California, during the Fall 2018 semester. In addition, T.W.\ and E.S.\
were partially supported by the Simons Foundation under Awards~581334
and~636383, respectively.

%
%
\bibliography{wanner1a,wanner1b,wanner2a,wanner2b,wanner2c}
\bibliographystyle{abbrv}
\end{document}